\documentclass[12pt]{amsart}

\usepackage[utf8]{inputenc}

\usepackage{hyperref}
\usepackage{graphicx}
\usepackage{amssymb}
\usepackage{float}
\usepackage{fullpage}
\usepackage{physics}
\usepackage{comment}
\usepackage{tikz}
\usetikzlibrary{positioning, arrows.meta, shapes.geometric}
\usepackage{mathtools}
\usepackage{bm}
\newtheorem{theorem}{Theorem}[section]
\newtheorem{lemma}[theorem]{Lemma}
\newtheorem{corollary}[theorem]{Corollary}
\newtheorem{proposition}[theorem]{Proposition}

\newcommand{\precprec}{\prec\mathrel{\mkern-5mu}\prec}

\theoremstyle{definition}
\newtheorem{definition}{Definition}

\theoremstyle{remark}

\numberwithin{equation}{section}

\newcommand{\K}{\mathrm{K}}
\newcommand{\G}{\mathrm{G}}

\newcommand{\C}{{\mathbb C}}
\newcommand{\R}{{\mathbb R}}

\newcommand{\Z}{{\mathbb Z}}

\newcommand{\eps}{\varepsilon}

\newcommand{\SL}{\operatorname{SL}}

\renewcommand{\H}{{\mathbb H}}
\renewcommand{\Re}{\operatorname{Re}}
\renewcommand{\Im}{\operatorname{Im}}

\renewcommand{\H}{\mathcal{H}}

\renewcommand{\aa}{\mathfrak{a}}
\newcommand{\n}{\mathrm{n}}

\newcommand{\g}{\mathrm{g}}
\newcommand{\bb}{\mathfrak b}
\newcommand{\cc}{\mathfrak c}
\newcommand{\gf}{\mathfrak g}
\renewcommand{\d}{ \mathrm d}
\renewcommand{\n}{\mathrm{n}}

\renewcommand{\pmod}[1]{\,\,(\mathrm{mod}\,{#1})}

\title{On the greatest prime factor and uniform equidistribution of  quadratic polynomials}

\author{Lasse Grimmelt}
\address{Mathematical Institute, University of Oxford\\ Radcliffe Observatory Quarter, Woodstock Rd\\
Oxford OX2 6GG\\
UK}
\email{lasse.grimmelt@maths.ox.ac.uk}
\author{Jori Merikoski}
\address{Mathematical Institute, University of Oxford\\ Radcliffe Observatory Quarter, Woodstock Rd\\
Oxford OX2 6GG\\
UK}
\email{jori.merikoski@maths.ox.ac.uk}
\date{}

\begin{document}
\begin{abstract}
We show that the greatest prime factor of  $n^2+h$ is at least $n^{1.312}$ infinitely often. This gives an unconditional proof for the range previously known under the Selberg eigenvalue conjecture. Furthermore, we get uniformity in $h \leq n^{1+o(1)}$ under a natural hypothesis on real characters. The same uniformity is obtained for the equidistribution of the roots of quadratic congruences modulo primes. We also prove a variant of the divisor problem for $ax^2+by^3$, which was used by the second author to give a conditional result about primes of that shape. 
\end{abstract}

\maketitle

\tableofcontents
\section{Introduction}
We consider quadratic polynomials of the shape $an^2+h$ and apply \cite{GMtechnical} by the authors, to obtain new results about their greatest prime factor, the equidistribution of their roots to prime moduli, and a certain divisor problem which is used in a recent work of the second author \cite{mx2y3}. 

A famous conjecture of Landau states that there are infinitely many primes of the form $n^2+1$. As an approximation, we consider the greatest prime factor of quadratic polynomials $an^2+h$ and show the following theorem that improves upon \cite{mlargepf} by the second author and \cite{pascadi2024large} by Pascadi in two aspects. First, we obtain unconditionally the  exponent $1.312$ that was previously known under the Selberg eigenvalue conjecture. Second, we get uniformity in the shift $h$ under an assumption involving the Legendre symbol $(\tfrac{x}{p})$. To state the assumption, we define
\begin{align*}
  \varrho_{a,h}(k):= \#\{ \nu \in \Z/k\Z: a\nu^2+h \equiv 0 \pmod{k} \}.
\end{align*}
This is a multiplicative function with $ \varrho_{a,h}(p) = 1 + (\frac{-ah}{p})$ for $\gcd(a,p)=1$. In contrast to all previous works, the proof is independent of progress towards the Selberg eigenvalue conjecture -- any fixed spectral gap would give the same quality of result.
\begin{theorem}\label{thm:primefactor}
There exists some small $\eps >0$ such that the following holds for all $X > \eps^{-1}$. Let $1 \leq h\leq X^{1+\eps}$ be square-free and let $1 \leq a\leq X^{\eps}$ with $\gcd(a,h)=1$. Suppose that for any $X^{\eps}<  Y < Z \leq X^2$ we have
\begin{align*}
  \sum_{Y \leq p < Z} \frac{\log p}{p} \varrho_{a,h}(p)\leq (1+\eps) \log Z/Y + 1/\eps.
\end{align*}
 Then there exists $n\in [X,2X]$ such that the greatest prime factor of $a n^2+h$ is greater than $X^{1.312}$.  
\end{theorem}
For $ah \leq X^{\eps^2}$ the hypothesis can be shown unconditionally by the classical zero-free region, since a possible Siegel zeros would help to make the sum small. It seems likely that this hypothesis is a fundamental obstruction, as its contraposition means that the $an^2+h$ have more  small prime factors than expected. We will not pursue this issue here, as our focus is the application of the automorphic methods developed in \cite{GMtechnical}.

In similar spirit, we can also obtain uniformity in $h$ for the equidistribution of roots of quadratic congruences to prime moduli, generalising the work of Duke, Friedlander, and Iwaniec \cite{DFIprimes}. We again need to include the distribution of $\varrho_{a,h}(p) $ in the statement. However, here the issue is of opposite nature and the theorem becomes trivial if $\varrho_{a,h}(p)$ is $0$ almost always, i.e. if there is a Siegel zero for the real character $(\frac{-ah}{\cdot})$. 
\begin{theorem}\label{thm:roots}
    Let $1 \leq h\leq X^{1+o(1)}$ be square-free and let $1 \leq a\leq X^{o(1)}$ with $\gcd(a,h)=1$. Then for any $0\leq \alpha <\beta \leq 1$ we have
    \begin{align*}
   \# \Bigl\{(p,\nu): p\leq X, \, a\nu^2+h\equiv 0 \pmod{p}, \, \frac{\nu}{p} \in ( \alpha, \beta] \Bigr\} =  (\beta-\alpha) \sum_{p \leq X} \varrho_{a,h}(p) + o\bigl(\pi(X)\bigr) .
    \end{align*}
\end{theorem}

Finally, in Section \ref{sec:x2y3} we state and prove a technical variant of the divisor problem for the binary polynomial $ax^2+by^3$. This was used by the second author \cite{mx2y3} to show the following result about primes of the shape $ax^2+by^3$, conditional to the hypothesis that certain Hecke eigenvalues exhibit square-root cancellation along the values of binary cubic forms. As is discussed in \cite[Section 1.2.3]{mx2y3}, a naive application of \cite{DFIprimes} does not provide a sufficient divisor function estimate for an asymptotic formula.
\begin{theorem}\label{thm:x2y3}   Let $a,b > 0$ be coprime integers. Assume that \cite[Conjecture $\mathrm{C}_a(\eps)$]{mx2y3} holds for all $\eps >0$.  Then 
    \begin{align*}
        \sum_{x \leq X^{1/2}} \sum_{y \leq X^{1/3}} \Lambda(ax^2+by^3) = (1+o(1)) X^{5/6}.
    \end{align*}
\end{theorem}

\subsection{Uniform equidistribution estimates}
By sieve methods, distribution of prime numbers may be reduced to so-called Type I and Type II sums.  The main technical results of this paper are the following two theorems that provide highly uniform asymptotics for the distribution of the roots of quadratic congruences. Theorems \ref{thm:primefactor} and \ref{thm:roots} are quick consequences of these by the sieve arguments in \cite{mlargepf} and \cite{DFIprimes}. It is convenient for us to define $A \precprec B$ to mean $|A| \leq X^{o(1)} B$. We let $\theta$ denote the best exponent towards the Selberg eigenvalue conjecture, so that by the work of Kim and Sarnak \cite{KimSarnak} $\theta \leq 7/64.$
\begin{theorem}
[Type I estimate] \label{thm:typeI} 
Let $1 \leq D \leq K \leq X^{2}$ and $D \leq X^{1/2}$. Let $h$ be square-free with $1 \leq h \precprec X^{2}$ and let $1 \leq a \precprec 1$ with $\gcd(a,h)=1$. Let $\psi_1,\psi_2:\R \to \C$ be smooth functions supported on $[1,2]$ and $[-1,1]$, resp., such that for all $J \geq 0$ we have $\psi_i^{(J)} \precprec_J 1$. Then
\begin{align*}
    \sum_{d \leq D}\bigg| \sum_{k \equiv 0 \pmod{d}} \psi_1\Bigl(\frac{k}{K}\Bigr) \Bigl(&\sum_{a \ell^2 +h \equiv 0 \pmod{k}} \psi_2\Bigl(\frac{\ell}{X}\Bigr)   - \frac{\varrho_{a,h}(k) }{k} X \int_\R \psi_2(u) \d u \Bigr) \bigg|\\
   & \precprec \,  D^{1/2}X^{1/2}(D^{1/2}+h^{1/4}) \Bigl(1+\frac{X}{D(D+h^{1/2})}\Bigr)^\theta.
\end{align*}
\end{theorem}

\begin{theorem}[Type II estimate]  \label{thm:typeII}
Let $M \geq N$ with $MN \geq X$ and $M \leq X$.  Let $h$ be square-free with $1 \leq h \precprec X^{2}$ and let $1 \leq a \precprec 1$ with $\gcd(a,h)=1$. Let $\psi$ be a smooth function supported on $[-1,1]$ such that for all $J \geq 0$ we have $\psi^{(J)} \precprec_J 1$.  Suppose that $\alpha_m,\beta_n$ are bounded coefficients with $\beta_n$ supported on square-free integers. Then
\begin{align*}
    \sum_{\substack{m\sim M }} &\sum_{n \sim N} \alpha_m \beta_n  \Bigl(\sum_{a \ell^2 +h \equiv 0 \pmod{mn}} \psi\Bigl(\frac{\ell}{X}\Bigr)   - \frac{\varrho_{a,h}(mn) }{mn}  X \int_\R \psi(u) \d u\Bigr) \\
   & \precprec \, M^{1/2} X^{1/2}  + M^{1/4}N^{1/2} X^{1/2} (N^{1/2} + h^{1/8} )\bigg(1+ \frac{X}{M^{1/2} N(N+h^{1/4}) }\bigg)^\theta.
\end{align*}
\end{theorem}

For $h \precprec D^2$ and $h \precprec N^4$, the Type I and Type II bounds respectively simplify to
\begin{align*}
 D X^{1/2}   \Bigl(1+\frac{X}{D^2}\Bigr)^\theta \quad  \text{and}\quad M^{1/2} X^{1/2}  + M^{1/4} N  X^{1/2} \bigg(1+ \frac{X}{M^{1/2} N^2 }\bigg)^\theta,
\end{align*}
which are independent of $h$. In the critical ranges for the proof of Theorem \ref{thm:primefactor} we have precisely $D \leq X^{1/2-o(1)}$ and $N \leq X^{1/4-o(1)}$, which ultimately allows us to take $h \leq X^{1+o(1)}$, and the bounds are then non-trivial as long as we have a spectral gap $\theta < 1/2$. Our Type I and II estimates remain non-trivial also for larger $h < X^2$ and one could in this range lower bound the size of the largest prime divisor by $X^{\varpi}$  for some $\varpi = \varpi(h) > 1$ with $\varpi \to 1$ as $h \to X^2$. We do not pursue this here further. 

\subsection{Outline of the proof} \label{sec:outline}
We consider Theorem \ref{thm:primefactor} for $a=1$ in this sketch. Similar to the work of the second author \cite{mlargepf}, a combination of Chebyshev's method and Harman's sieve leads us to consider the distribution of the roots of $\ell^2+h \equiv 0 \pmod{k}$ in a narrow window $\ell \sim X$  in $\Z/k \Z$ with $X < K$. More precisely, we want to evaluate asymptotically
\begin{align*}
&(\text{Type I sums})  &\sum_{d \sim D} \lambda_d \sum_{\substack{k \sim K \\ k \equiv 0 \pmod{d}}} \sum_{\substack{\ell \sim X \\ \ell^2+h \equiv 0 \pmod{k}}} 1 \quad \text{and} \\
&(\text{Type II sums}) &\sum_{m \sim M} \alpha_m \sum_{\substack{n \sim N}} \beta_n \sum_{\substack{\ell \sim X \\ \ell^2+h \equiv 0 \pmod{mn}}} 1, \quad  \quad
\end{align*}
for as wide range of $D$ and $N$ as possible, see Theorems \ref{thm:typeI} and \ref{thm:typeII}.  The case of Type II sums may be reduced to a more complicated variant of the Type I sums by an application of Cauchy-Schwarz, and so we shall first focus on Type I sums in this sketch. 

We note that all previous results only consider bounded $h.$ Hooley \cite{hooley} estimated the Type I sums by using the Gauss correspondence to connect them to Kloosterman sums and bounding them with Weil bound. Combining this with an upper bound sieve, Hooley  obtained the exponent 1.1 in Theorem \ref{thm:primefactor}. Deshouillers and Iwaniec \cite{DIprimefactor} improved this to 1.2 by making use of sums of Kloosterman sums \cite{deshoullieriwaniec}, which was improved to 1.2182 by de la Bretèche and Drappeau \cite{BD} by sharpening the dependency on exceptional eigenvalue. In the work of the second author \cite{mlargepf} Harman's sieve and Type II sums were introduced to the proof, improving  the exponent to 1.279 unconditionally and 1.312 conditional to Selberg's eigenvalue conjecture. Finally, the dependency on the exceptional eigenvalue was recently improved by Pascadi \cite{pascadi2024large} who achieved 1.3. The last two results assume that $h=1$.

In this article, we will take a different approach to estimating such sums, originating in the work of Duke, Friedlander, and Iwaniec \cite{DFIprimes}. Following their approach, we consider symmetric matrices $\gf$ of determinant $h$
\begin{align*}
\det( \gf)= \det \mqty(m & \ell \\ \ell & k) = mk -\ell^2 = h.
\end{align*}
The modular group $\SL_2(\Z)$ acts on symmertic matrices via $\gf \mapsto \gamma \gf \gamma^t$, and the congruence $k \equiv 0 \pmod{d}$ is preserved by the Hecke congruence subgroup $\Gamma_0(d)$. For a fixed modulus $d$ the sum can be parametrized as an average over a Poincaré type series with a smooth weight $F$ encoding $k \sim K$ and $\ell \sim L$, that is,
\begin{align*}
    \sum_{\substack{k \sim K \\ k \equiv 0 \pmod{d}}} \sum_{\substack{\ell \sim X \\ \ell^2+h \equiv 0 \pmod{k}}} 1 = \sum_{z \in \Lambda_h} \sum_{\tau \in  \Gamma_0(d) \backslash \SL_2(\Z)} \alpha_{d,h}(\tau z)\sum_{\gamma \in \Gamma_0(d)} F(\gamma \tau z).
\end{align*}
Here $z$ ranges over the Heegner points $\Lambda_h$ in the upper half-plane $\H$ and for each $z$ the weight $\alpha_{d,h}(\tau z)$ restricts $\tau$ to $d^{o(1)}$ cosets. In \cite{DFIprimes} such sums are estimated by using the spectral theory of automorphic forms for each fixed $\tau$, $z$, and $d$ separately. This is costly if $h$ is large since $\#\Lambda_h = h^{1/2+o(1)}$. Therefore, to get uniformity in $h$, we make use of the averages over $z$, $\tau$, as well as $d$.

In \cite{GMtechnical} the authors proved a result (cf. Theorem \ref{thm:technical} below) that allows one to count sums of this shape. More generally, given a smooth function $F:\SL_2(\R) \to \C$ and functions $\alpha_1,\alpha_2:\SL_2(\R) \to \C$ supported on some finite sets, we can estimate the \emph{weighted average discrepancy} $\langle \alpha_1 | \Delta_d F | \alpha_2 \rangle $ defined by
\begin{align*}
\langle \alpha_1 | \Delta_d F | \alpha_2 \rangle = \sum_{\tau_1,\tau_2} \overline{\alpha(\tau_1)}\alpha(\tau_2) \Bigg(\sum_{\gamma \in \Gamma_0(d)} F(\tau_1^{-1} \gamma \tau_2)  - \frac{1}{|\Gamma_0(d) \backslash \SL_2(\R)|} \int_{\SL_2(\R)} F(\g )  \d \g \Bigg).
\end{align*}
Note that we have to normalize the symmetric matrices $\gf$ by a factor of $h^{-1/2}$ to get determinant $1$. Letting $\alpha_1=I$ denote the point weight at the identity $\smqty(1 & \\ & 1)$,  Theorem \ref{thm:technical} then says that for any choice of $Z_0Z_1Z_2 = X h^{-1/2}$ we have
\begin{align} \label{eq:introtechnical}
    \langle I | \Delta_d F | \alpha_{d,h} \rangle \precprec (X h^{-1/2})^{1/2} Z_0^{\theta} \sqrt{  \langle I | \Delta_d k_1 | I \rangle  \langle \alpha_{d,h}  | \Delta_d k_2 | \alpha_{d,h} \rangle},
\end{align}
where $k_i:\SL_2(\R) \to [0,1]$ are certain functions supported on matrices of size $Z_i^2$. Applying Cauchy-Schwarz on $d$ we get
\begin{align*}
   \sum_{d \sim D} \bigl|\langle I | \Delta_d F | \alpha_{d,h} \rangle \bigr| \precprec (X h^{-1/2})^{1/2} Z_0^{\theta} \sqrt{    \sum_{d \sim D}  \langle I | \Delta_d k_1 | I \rangle    \sum_{d \sim D}  \langle \alpha_{d,h}  | \Delta_d k_2 | \alpha_{d,h} \rangle},
\end{align*}
and the goal is to bound the right-hand side by absorbing the sum over $d$ via a divisor bound. We emphasise that in this argument the existing divisor switching symmetry for the $d$ variable is fully preserved throughout the application of spectral methods.  This is the reason that ultimately allows us to remove the dependency on exceptional eigenvalues here as well as for the Type II sums.

As usual after Cauchy-Schwarz, we need to consider diagonal and off-diagonal contributions when bounding the terms $ \sum_d \langle \alpha_{i}  | \Delta_d k_i | \alpha_{i} \rangle$. The new kernels $\langle \alpha_{i}  | \Delta_d k_i | \alpha_{i} \rangle$ may be written explicitly in terms of a sum over certain pairs of integer matrices $(\g_1,\g_2)$ with a congruence condition of the shape 
\begin{align}\label{eq:introcong}
    F(\g_1,\g_2) \equiv 0 \pmod{d},
\end{align}
for some polynomial $F$ in the entries of the matrices $\g_1$ and $\g_2$. The diagonal is then given by the set of pairs $(\g_1,\g_2)$ where $F(\g_1,\g_2)  = 0$ over $\Z$. There the congruence \eqref{eq:introcong} is fulfilled trivially and the sum over $d$ results in a factor of $D$. However, we can make use of the fact that the diagonal is a sparse subset. In the off-diagonal, the congruence condition \eqref{eq:introcong} is non-trivial and we can use a divisor bound to absorb the $d$ summation. For the Type I sums we are then able to show that essentially (cf. Propositions \ref{prop:heegner} and \ref{prop:lowertriang})
\begin{align*}
     \sum_{d \sim D}  \langle \alpha_{d,h}  | \Delta_d k_2 | \alpha_{d,h} \rangle &\precprec D h^{1/2} + h Z_2 \\
      \sum_{d \sim D}  \langle I | \Delta_d k_1 | I \rangle & \precprec D + Z_1 +\frac{K}{X}.
\end{align*}
Choosing $Z_1=D$ and $Z_2 = D h^{-1/2}$ balances the terms and gives (for $h \ll D^2$)
\begin{align*}
     \sum_{d \sim D} \alpha_d \langle I | \Delta_d F | \alpha_{d,h} \rangle \precprec D X^{1/2} \Bigl( 1+ \frac{X}{D^2} \Bigr)^\theta.
\end{align*}
This is non-trivial in the range $D < X^{1/2}$. It is noteworthy that, while in the new kernel we subtract the main term $\int (\cdots) \d g$ by the definition of $\Delta_d$, we do not make any use of this and simply drop the integral by positivity. The new ranges $Z_i^2$ are too short for us to show cancellation.

For the Type II sums, we apply Cauchy-Schwarz to smoothen $m$, which gives us a sum over pairs of symmetric matrices with the same top left entry $m$, and congruence conditions modulo $n_1$ and $n_2$, respectively. We give a parametrisation for such pairs in Lemma \ref{lem:para2}.  In place of $\alpha_1=I$  we then need to consider more general weights $\beta_{n_1,n_2}$ that are sums over lower triangular matrices.  Roughly speaking, for $\beta_{n_1,n_2}$ being the indicator of matrices $\smqty( 1 & \\ x & 1)$ with $x = n_1v\,\overline{n_1}$ for $v \leq V:=X/M$ and $n_1 \overline{n_1} \equiv 1 \pmod{n_2}$, we need to bound weighted average discrepancy of the shape
\begin{align*}
 \sum_{\substack{ n_1,n_2 \sim N \\ \gcd(n_1,n_2)=1}}  \langle \beta_{n_1,n_2} | \Delta_{n_1n_2} F | \alpha_{h,n_1n_2} \rangle.
\end{align*}
This is achieved by Theorem \ref{thm:technical} combined with Propositions \ref{prop:heegner} and \ref{prop:lowertriang}, where the kernel associated to $\alpha_{h,d}$ is bounded similarly to the Type I case. For practical purposes we get for the other kernel
\begin{align*}
     \sum_{\substack{ n_1,n_2 \sim N \\ \gcd(n_1,n_2)=1}} \langle \beta_{n_1,n_2} | \Delta_{n_1n_2} k_1 | \beta_{n_1,n_2} \rangle \precprec   V N^2+V^2 Z_1^{1/2}.
\end{align*}
Choosing $Z_1= N^2/V^2$, for $MN=X^\alpha$ this suffices for the Type II range $X^{\alpha-1} < N < X^{(2-\alpha)/3}$. This range was obtained in \cite{mlargepf} under Selberg's eigenvalue conjecture and for the fixed choice $h=1$.

To compare with the classical approach, if we applied sums of Kloosterman sums and the Kuznetsov formula, we would have to provide an exceptional spectral large sieve bound for
\begin{align*}
  \sum_{d \leq D}  \sum_{\lambda_j=1/4-\nu_j^2}^{\Gamma_0(d)} Z^{2\nu_j} \bigg| \sum_{n \leq N} a_{n,d}  \varrho_j (n)\bigg|^2.
\end{align*}
Here the  coefficients $a_{n,d}$ are intricately entangled with the level $d$, especially in the case of Type II sums. In \cite{mlargepf} the results of Deshouillers and Iwaniec \cite{deshoullieriwaniec} and de la Bretèche and Drappeau \cite{BD} were used as a black box and the coefficients $a_{n,d}$ are essentially treated as arbitrary in some cases. For coefficients $a_{n,d}$ which do not depend on the level $d$, Deshouillers and Iwaniec \cite{deshoullieriwaniec} improve the dependency on the exceptional eigenvalue by a divisor switching argument for the sums of Kloosterman sums. However, our divisor switching completely in physical space turns out to be more effective. Recently, Pascadi \cite{pascadi2024large} improved the spectral large sieve for a fixed modulus when $a_{n,d}$ are  exponential phases $e(n \theta_d)$, or more general dispersion type coefficients. From the perspective of the spectral large sieve, our application of \cite{GMtechnical} allows us to fully capture the entanglement of $a_{n,d}$ with the level $d$.

We have restricted to the case of a negative discriminant, that is, $a,h \geq 1$. It would be interesting to consider positive discriminants, combining our ideas with the approach in \cite{toth}.

\subsection{Notations}
We follow standard asymptotic Vinogradov and Landau notations. It is convenient for us to define $A \precprec B$ to mean $|A| \leq X^{o(1)} B$, where $X$ is the largest scale in the given context. For example, for $n \leq X$ we will frequently use the divisor bound $d(n) \precprec 1$.

For matrices we use the notations from \cite{GMtechnical}. We let $\G = \SL_2(\R)$ and denote the subgroups
\begin{align*}
    \mathrm{N} :=& \bigg\{ \mathrm{n}[x] =  \mqty(1&x \\ & 1) : x \in \R  \bigg\}  \\
      \mathrm{A} := & \bigg\{ \mathrm{a}[y]= \mqty(\sqrt{y }& \\ & 1/\sqrt{y}) : y \in (0,\infty)  \bigg\} \\
        \mathrm{K} := &\bigg\{ \mathrm{k}[\theta] =  \mqty(\cos \theta & \sin \theta \\ -\sin \theta & \cos \theta) : \theta \in \R / 2\pi \Z \bigg\}  
\end{align*}
We denote $\g = \smqty(a & b \\ c& d) \in \G$ and  use fraktur symbols for symmetric matrices $\gf = \smqty( \aa & \bb \\ \bb& \cc)$.

\section{Technical result}
We now import a technical result from \cite{GMtechnical} that bounds the weighted average discrepancy as in \eqref{eq:introtechnical}. Recall the Hecke congruence subgroup of level $q \geq 1$
\begin{align*}
 \Gamma= \Gamma_0(q)  := \Bigl\{ \mqty(a & b \\ c& d) \in \SL_2(\Z) : c \equiv 0 \pmod{q} \Bigr\}.
\end{align*}
We let $\Gamma$  act on the Lie group $\G := \SL_2(\R)$ by multiplication from the left $\g \mapsto \gamma \g$.  Given a compactly supported function $F:\G \to \C$ we construct the \emph{automorphic kernel function} $\mathcal{K}_{q}F:\G \times \G \to \C$ via
\begin{align} \label{eq:kerneldef}
      ( \mathcal{K}_{q} F)(\tau_1,\tau_2)&=\sum_{\gamma \in \Gamma}  F(\tau_1^{-1} \gamma \tau_2).
\end{align}

If $F$ is a smooth bump function and its support is large  and not too skewed compared to the level $q$, we expect that the sum over $\Gamma$ is well approximated by the corresponding integral over $\G$ against the Haar measure $\d \g$, normalized by the volume of the fundamental domain $\Gamma \backslash \G$. We then define the local discrepancy
\begin{align*}
  \Delta_q F (\tau_1,\tau_2)  &:=    \mathcal{K}_qF(\tau_1,\tau_2) - \frac{1}{|\Gamma_0(q) \backslash \G|} \int_\G F( \g ) \d \g.
\end{align*}

We let $\mathcal{L}(\G)$  denote the set of linear functionals $C(\G) \to \C$. For $\alpha \in \mathcal{L}(\G)$ we will denote the image of $ f \in C(\G)$ by $\langle f\rangle_\alpha  := \alpha(f) \in \C$. 
\begin{definition}[Compactly supported linear functional]
We define the set of compactly supported linear functionals via
\begin{align*}
    \mathcal{L}_c(\G) := \bigcup_{K \subseteq \G \, \text{compact}} \bigcup_{c>0}\Bigl\{\alpha \in \mathcal{L}(\G):     |  \langle f \rangle_\alpha| \leq c \sup_{\g \in K}  |f(\g)| \text{ for all } f \in C(\G)\Bigr \}.
\end{align*}
\end{definition}
In this paper we are interested only in functionals given by weighted sums $f \mapsto \sum_{\tau \in T}  \alpha(\tau)f(\tau)$ for some finite set $T \subseteq 
 \G$ and weight $\alpha:T \to \C$. For any $\alpha \in \mathcal{L}(\G)$ we define the complex conjugate by 
\begin{align*}
  \langle f \rangle_{\overline{\alpha}} =  \overline{ \langle \overline{f} \rangle_{\alpha} } .
\end{align*}
\begin{definition}[Sesquilinear form induced by a binary function]
For any  $F \in C(\G \times \G)$ we define the sesquilinear form  $\mathcal{L}_c(\G) \times \mathcal{L}_c(\G) \to \C$ 
\begin{align*}
 \langle \alpha_1 | F | \alpha_2 \rangle := \langle \langle F \rangle_{\overline{\alpha_1}} \rangle_{\alpha_2},
\end{align*}
where $\overline{\alpha_1}$ acts on the left variable of $F(\g_1,\g_2)$ and $\alpha_2$  acts on the right variable.
\end{definition}
Restriction to $\mathcal{L}_c(\G)$ means that this is well-defined, that is, $\langle F \rangle_{\overline{\alpha_1}}$ defines a continuous function in the second variable, and we have $\langle F \rangle_{\overline{\alpha_1}} \rangle_{\alpha_2} = \langle \langle F \rangle_{\alpha_2} \rangle_{\overline{\alpha_1}},$  where in both expressions $\overline{\alpha_1}$ acts on the left variable and $\alpha_2$  acts on the right variable. This may be viewed as a tensor product of $\overline{\alpha_1}$ and $ \alpha_2$, that is, for $F(\tau_1,\tau_2) =\overline{f_1(\tau_1)} f_2(\tau_2)$ we have
\begin{align*}
   \langle \alpha_1 | F | \alpha_2 \rangle =  \overline{\langle f_1\rangle_{\alpha_1}}   \langle f_2\rangle_{\alpha_2}.
\end{align*}

For the proof of Theorem \ref{thm:x2y3} we require Hecke operators, for Theorems \ref{thm:primefactor} and \ref{thm:roots} this is not needed. For any integer $h \geq 1$ we denote the set of Hecke orbits
\begin{align}
    \label{eq:Hecke} H_h = \Bigl\{ \mqty(e & f \\ & g): \quad eg = h, \quad f \in \Z/g \Z \Bigr\}
\end{align}
and recall that the set of all integer matrices with determinant $h$ may be parametrised by $\SL_2(\Z)$ as a disjoint union
\begin{align*}
    \mathrm{M}_{2,h}(\Z) = \bigsqcup_{\sigma \in H_h}\SL_2(\Z) \sigma =  \bigsqcup_{\sigma \in H_h} h \sigma^{-1}\SL_2(\Z).
\end{align*}
We then define the Hecke operator acting on $f:\G \to \C$ by
\begin{align} \label{eq:heckeopdef}
    \mathcal{T}_h f(g) :=  \frac{1}{\sqrt{h}} \sum_{\sigma \in H_h} f\bigl(\tfrac{1}{\sqrt{h}}\sigma g\bigr).
\end{align} 
For any function $f:\G\times \G\to \C$ we define the two variables Hecke operators
\begin{align*}
    (\mathcal{T}_{h_1,h_2} f)(\g_1,\g_2)= (\mathcal{T}_{h_1})_{\g_1}(\mathcal{T}_{h_2})_{\g_2} f(\g_1,\g_2),
\end{align*}
where $(\mathcal{T}_{h})_{\g_i}$ is the Hecke operator $\mathcal{T}_{h}$ acting on the $\g_i$ coordinate.

We define  for any $R > 0$  and  $\g = \smqty(a &b \\ c& d) \in \G$ the $R$-skewed  hyperbolic size 
\begin{align*}
    u_R(\g)  := \tfrac{1}{4} (a^2+(b/R)^2+(cR)^2+d^2-2).
\end{align*}
We denote  $\theta=\theta(\Gamma) := \max\{0,\Re(\sqrt{1/4-\lambda_1(\Gamma)})\}$ with $\lambda_1(\Gamma)$ denoting the smallest positive eigenvalue of the hyperbolic Laplacian for $\Gamma$. The Selberg eigenvalue conjecture states that
 we have $\theta=0$ and the best unconditional bound for congruence subgroups is $\theta \leq 7/64$ by Kim and Sarnak \cite{KimSarnak}. The following is Theorem 8.1 in \cite{GMtechnical}.
\begin{theorem} \label{thm:technical}
   Let $q \geq 1$ and let $\mathbf{X},\mathbf{Y} > 0$ and $\delta \in (0,1)$ and suppose that
  $\mathbf{X}/\mathbf{Y}  > \delta$.  let $f(x,y)$ be a smooth function, supported on $[-1,1]\times[1,2]$ and satisfying $\partial_x^{J_1} \partial_y^{J_2} f \ll_{J_1,J_2} \delta^{-J_1-J_2}$. Denote $F:\G \to \C, \, F(\smqty(a & b \\ c& d)):= f(\frac{x}{\mathbf{X}},\frac{y}{\mathbf{Y}})$ with $y=1/(c^2+d^2)$ and $x= y (ac+bd)$.   Let $H \geq 1$ and let $\beta(h)$ be complex coefficients supported on $h \leq H$. Let $\alpha_1,\alpha_{2,h} \in \mathcal{L}_c(\G)$.

Then for any choice $Z_0Z_1Z_2 \geq \mathbf{X}/\mathbf{Y}+1$ with $Z_0,Z_1,Z_2 \geq 1$ we have 
\begin{align*}
 \sum_{\substack{\gcd(h,q)=1}}\beta(h)  \langle \alpha_1 | \mathcal{T}_{h,1}  \Delta_q F | \alpha_{2,h} \rangle \ll & \, q^{o(1)} \delta^{-O(1)}  (\mathbf{X}/\mathbf{Y})^{1/2+o(1)} H^{1/2} Z_0^\theta  \, \langle  \alpha_1| \Delta_q k_{Z_1^2,\mathbf{X}}| \alpha_1\rangle^{1/2} \\
& \times\bigg(  \, \sum_{h \leq H} |\beta(h)|^2\langle  \alpha_{2,h} |\Delta_q k_{Z_2^2,1} | \alpha_{2,h} \rangle \bigg)^{1/2},
\end{align*}
for certain smooth functions $k_{Y,R} :\G \to [0,1]$ which satisfy
\begin{align*}
    k_{Y,R}(\g)  \leq  \frac{\mathbf{1}\{u_R(\g) \leq Y\}}{\sqrt{1+u_R(\g)}}.
\end{align*}
\end{theorem}

Even though the original weight function $F$ factorizes through the upper half-plane $\mathcal{H} \cong \G/\K$, using the group $\G$ as the ambient space is illuminating. Indeed, after Cauchy-Schwarz we get two automorphic kernel functions, one that lives properly on $\G$ and one that lives on $\K \backslash \G / \K$.

\section{Parametrisation of  symmetric matrices}
The counting problems we are interested in are related to symmetric matrices with congruence conditions. In this section, building up on  \cite[Section 2]{DFIprimes}, we parametrise these by Hecke congruence subgroups.

\subsection{The base case}
Define the set of symmetric integer matrices with determinant $h$
\begin{align*}
    \mathrm{S}_h := \Bigl\{ \gf=\mqty( \aa  & \bb \\ \bb & \cc) \in \mathrm{M}_{2,h}(\Z): \, \aa,\cc > 0 \Bigr\}
\end{align*}
and define for $a,d\geq 1$, $\gcd(a,h)=1$
\begin{align*}
     \mathrm{S}_{a,h}(d) := \Bigl\{ \mqty( \aa  & \bb \\ \bb & \cc) \in    \mathrm{S}_{ah}: \cc \equiv 0 \pmod{ad}, \, \bb\equiv 0 \pmod{a},\, \aa,\cc > 0 \Bigr\}.
\end{align*}
By dividing $\bb^2-\aa \cc = a h$ throughout by $a$, we see that $\mathrm{S}_{a,h}(d)$ corresponds bijectively to the set 
\begin{align*}
    \Bigl\{ (m,\ell,k) \in \Z^3: a\ell^2 - m k = h, \quad   k \equiv 0 \pmod{d} , \, m,k > 0\Bigr\},
\end{align*}
with $(\aa,\bb,\cc) =(m,a\ell,ak)$. Thus, for the smooth weights that appear in the proofs, we have under this identification $F(m,\ell,k) = F(\aa,\bb/a,\cc/a)$.

We can identify  $\mathrm{S}_h$ with the set of points in the upper half-plane $\mathcal{H}$
\begin{align*}
   \{ z= \frac{\bb + i \sqrt{h}}{\cc} \in \mathcal{H}: \aa,\bb,\cc \in \Z, \, \bb^2+h = \mathfrak{a} \cc \}. 
\end{align*}
In this way, the usual action of $\gamma \in \SL_2(\Z)$ on $\H$ by M\"obius transformations corresponds to the action on $\mathrm{S}_h$ defined by
\begin{align*}
    \gf \mapsto  \gamma\diamond\gf := \gamma \gf \gamma^t.
\end{align*}

We let $\mathcal{F}  \subseteq \H$ denote the standard fundamental domain for $\SL_2(\Z) \backslash \H$. We will denote the set of Heegner points of discriminant $h$ by
\begin{align*}
     \Lambda_h := \bigl\{ z= \frac{\bb + i \sqrt{h}}{\cc} \in \mathcal{F}: \aa,\bb,\cc \in \Z, \, \bb^2+h = \mathfrak{a} \cc \Bigr\}. 
\end{align*} 
Then
\begin{align}\label{eq:Heegnerhull}
     \Lambda_h  \subseteq \Bigl\{\frac{\bb + i \sqrt{h}}{\cc} :\aa,\bb,\cc \in \Z, \,\bb^2+h = \aa \cc, \quad \cc \leq 2 \sqrt{h}, \quad |\bb| \leq \frac{\cc}{2}\Bigr\}.
\end{align}
In particular, this implies that  $\# \Lambda_h  \leq h^{1/2+o(1)}$.
With this, a set of representatives for $\mathrm{S}_h$ acted upon by $\SL_2(\Z)$ is given by $\sqrt{h}\, \sigma \sigma^t$ with
\begin{align*}
  \sigma \in \mathrm{L}_h :=  \Bigl\{ \sigma = \smqty(\ast & \ast \\ 0 &\ast) \in \SL_2(\R): \sigma i \in \Lambda_h  \Bigr\}. 
\end{align*}

For $\tau = \smqty(\ast & \ast \\ c_0 & d_0) \in \SL_2(\Z)$ and $\gf = \smqty(\aa & \bb \\ \bb & \cc) \in \mathrm{S}_h$ we have 
\begin{align} \label{eq:cctransform}
    \cc(\tau \diamond \gf) = c_0^2 \aa + 2 c_0 d_0 \bb + d_0^2 \cc.
\end{align}
For $z= \g i$ corresponding to the symmetric matrix $\gf=\g\g^t$, we denote by abuse of notation $\cc(z) = \cc(\gf)$ and similarly for $\aa$ and $\bb$. 

For any $q \geq 1$ the congruences $\cc \equiv 0 \pmod{aq}$ and $\bb \equiv 0 \pmod{a}$ are fixed by the action of the Hecke congruence subgroup $\Gamma_0(q)$ of level $q=ad$. The above discussion gives the following parametrisation.
\begin{lemma}\label{lem:para}
Let $\gcd(a,h)=1$. Let $\Gamma=\Gamma_0(q)$ with $q=ad$ and let $\Gamma_z$ denote the stabilizer of a point $z \in \H$.  Let $T_q= \Gamma \backslash \SL_2(\Z)$. Then the determinant normalised set $(ah)^{-1/2}\mathrm{S}_{a,h}(d)$ can be written as a disjoint union by
    \begin{align*}
(ah)^{-1/2}\mathrm{S}_{a,h}(d)=\bigsqcup_{\sigma\in \mathrm{L}_{ah}} \bigsqcup_{\substack{\tau\in T_q \\ \cc(\tau \sigma i) \equiv 0 \pmod{q} \\ \bb(\tau \sigma i) \equiv 0 \pmod{a}}}  \Gamma_{\sigma i}\backslash \Gamma \, \tau \sigma \, \diamond \, \smqty(1 & \\ & 1).
    \end{align*}
\end{lemma}

\subsection{Type II sums}
We now consider a certain set of pairs of symmetric matrices that appear with the Type II sums after 
an application of Cauchy-Schwarz. We reduce them to $\mathrm{S}_{a,h}$ from the previous subsection. For $\gcd(n_1,n_2)=1$ consider
\begin{align*}
    \mathrm{S}^{(2)}_{a,h}(s,n_1,n_2)=\{(\mathfrak{g}_1,\mathfrak{g}_2) \in    \mathrm{S}_{a,h}(sn_1) \times  \mathrm{S}_{a,h}(sn_2) : \aa_1=\aa_2, \bb_1\equiv \bb_2 \pmod{s \aa_1}  \}. 
\end{align*}
We now show that one can make $\mathfrak{g}_1$ and $\mathfrak{g}_2$ equal under the action of lower triangular matrices $\n[x]^t = \smqty(1 & \\x & 1)$. By the definition of the action, we have
\begin{align*}
   \n[x]^t \diamond \mathfrak{g} = \mqty(1 & \\ x&1   ) \mqty ( \aa & \bb \\ \bb & \cc ) \mqty(1 & x\\ &1   ) =   \mqty ( \aa & \bb + x\aa \\ \bb + x\aa & \ast ).
\end{align*}
Denote $\aa=\aa_1=\aa_2$. The divisibility condition $\cc_i \equiv 0 \pmod{a s n_i}$ is fixed for $x_i \equiv 0 \pmod{a s n_i}$.  Recall that $\gcd(n_1,n_2)=1$, $\bb_1 \equiv \bb_2 \pmod{ s \aa}$, $\bb_1 \equiv \bb_2\equiv 0 \pmod{a}$, and by $\gcd(a,h)=1$  we have $\gcd(a,\mathfrak{a}sn_i)=1$. Thus, by Bezout, there exist $x_i\equiv 0 \pmod {a s n_i}$ such that $\n[x_1]^t \diamond \mathfrak{g}_1= \n[x_2]^t \diamond \mathfrak{g}_2$. Moreover, denoting $x_j= a s n_j u_j$, the equality $\n[x_1]^t \diamond \mathfrak{g}_1= \n[x_2]^t \diamond \mathfrak{g}_2$ is invariant under $(u_1,u_2) \mapsto (u_1+n_2,u_2+n_1)$. We summarise our findings in the following parametrisation.
\begin{lemma}\label{lem:para2}
Let $\gcd(a,h)=\gcd(n_1,n_2) =1$ and define
\begin{align*}
  U = \Z^2/(n_2, n_1) \Z, \text{ where } (n_2, n_1) \Z = \{(n_2 k,n_1 k): k \in \Z\}.  
\end{align*}
  Then we have a disjoint union
   \begin{align*}
     \mathrm{S}^{(2)}_{a,h}(s,n_1,n_2) =  \bigsqcup_{(u_1,u_2) \in U} \Bigl\{ (\n[a sn_1 u_1]^t \diamond \gf, \n[a sn_2 u_2]^t \diamond \gf   ): \gf \in  \mathrm{S}_{a,h}(sn_1n_2)\Bigr\}.
\end{align*} 
\end{lemma}

\subsection{Cubic discriminants}
The proof of Theorem \ref{thm:x2y3} involves discrimants of the form $b y^3$. The following lemma shows how one can use the Hecke orbits from \eqref{eq:Hecke} and $S_{a,h}$ to parametrise symmetric matrices with determinant $hy^2$ where $y|h$.
\begin{lemma}\label{lem:paracube}
Let $y|h$, $\gcd(a,h)=1$, and suppose that $\gcd(h,y^2)$ is square-free. Let $\Gamma=\Gamma_0(q)$ with $q=ad$ and let $\Gamma_z$ denote the stabilizer of a point $z \in \H$. Suppose that $\gcd(y,q)=1$. Then the set $\mathrm{S}_{a,hy^2}(d)$ can be written as a disjoint union by
    \begin{align*}
\mathrm{S}_{a,h y^2}(d)=\bigsqcup_{\sigma\in H_{y}} (y \sigma^{-1})  \diamond \mathrm{S}_{a,h}(d).
    \end{align*}
\end{lemma}
\begin{proof}
    It suffices to show that for any $\gf= \smqty( \aa  & \bb \\ \bb & \cc) \in \mathrm{S}_{a,h y^2}(d)$ there exists a unique $\sigma \in H_y$ such that
    \begin{align*}
        (y^{-1} \sigma) \diamond     \mqty( \aa  & \bb \\ \bb & \cc)  \in  \mathrm{S}_{a,h}(d).
    \end{align*}
We compute for $\sigma = \smqty(e & f \\ & g)$  with $eg=y$ and $f \pmod{g}$
\begin{align*}
       (y^{-1} \sigma) \diamond     \mqty( \aa  & \bb \\ \bb & \cc) = y^{-2} \mqty (e^2 \aa + 2 e f \bb + f^2 \cc & y \bb+g f \cc \\ y \bb+g f \cc & g^2 \cc ) =:  \mqty( \aa'  & \bb' \\ \bb' & \cc') .
\end{align*}
This is a matrix with real entries and determinant $h$. If we show that the entries are integral, then the congruence conditions $\bb' \equiv 0 \pmod{a}$ and $\cc' \equiv 0 \pmod{ad}$ follow since $\bb \equiv 0 \pmod{a}$, $\cc \equiv 0 \pmod{ad}$, and $\gcd(y,ad)=1$.

We first show that the matrix can be in $S_{a,h}(d)$ only if  $e|y$ is the maximal divisor such that $e^2 | \cc$. Indeed, if $ey_1|y$ satisfies $(e y_1)^2 | \cc$, then  $y_1^2 | \cc'$. By $y_1|h$ and $\bb'^2-\aa'\cc' = h$ this implies $y_1|\bb'$. But then, since $\bb'^2-\aa'\cc' = h$, we have $y_1^2 | h$. This contradicts the assumption that $\gcd(h,y^2)$ is square-free.

Assume now that $e|y$ is maximal such that $e^2 | \cc$. We complete the proof of the Lemma by showing that there is a unique $f \pmod {g}$ that satisfies
\begin{align*}
 \begin{cases}
  e^2 \aa + 2 e f \bb + f^2 \cc \equiv 0 &\pmod{y^2}  \\ y \bb+g f \cc \equiv 0 &\pmod{y^2} \\ g^2 \cc \equiv 0 &\pmod{y^2}.
\end{cases}  
\end{align*}
The third congruence is satisfied by $g= y/e$ and $e^2|c$. Denoting $y=ey_1$, $\bb=e\bb_1$, $\cc=e^2 \cc_1$, the first two congruence conditions are equivalent to
\begin{align} \label{eq:congruencepair}
 \begin{cases}
   \aa + 2  f \bb_1 + f^2 \cc_1 \equiv 0 &\pmod{g^2}  \\  \bb_1+ f \cc_1 \equiv 0 &\pmod{g} .
\end{cases} 
\end{align}
By $g|y|h$ we have
\begin{align*}
  \bb_1^2-\aa_1 \cc_1=g^2 h\equiv 0 \pmod {g^3}.
\end{align*}
and by the maximality of $e$ we have  that $(\cc_1,g)$ is square-free. Therefore,
\begin{align*}
    \cc_1(f + \frac{\bb_1}{\cc_1})^2 = \cc_1^2 + 2 \bb_1 f + \frac{\bb_1^2}{\cc_1} \equiv 2  f \bb_1 + f^2 \cc_1 + \aa  \pmod{g^2},
\end{align*}
where we made use of the fact that $\frac{\bb_1}{\cc_1} $ is well defined $\pmod{g}$, since  $(g,\cc_1)|\bb_1$. Thus, $f \equiv -\frac{\bb_1}{\cc_1} \pmod{g}$ is the unique solution to the pair of equations \eqref{eq:congruencepair}. 
\end{proof}
\section{Upper bounds for automorphic kernels} In this section we  bound the kernels that arise from an application of Theorem \ref{thm:technical}. By making use of an average over the level $q$, we obtain the essentially optimal bounds. We need two types of bounds. The first considers functionals defined by averages over Heegner points (cf. Lemma \ref{lem:para}), and the second functionals defined by averages over lower-triangular matrices (cf. Lemma \ref{lem:para2}).
\subsection{Heegner points}
Recall the definition \eqref{eq:kerneldef} of the automoprhic kernel $\mathcal{K}_q F$ for the group $\Gamma_0(q)$, and that we use $A \precprec B$ to mean $|A| \leq X^{o(1)} B$, where $X$ is the largest scale in the given context
\begin{proposition} \label{prop:heegner}
 Let $h,Q,Z \geq 1$. For $\Gamma= \Gamma_0(q)$ and $T_q= \Gamma \backslash \SL_2(\Z)$, let $\alpha_q = \alpha_{q,h}$ be the linear functional defined by
 \begin{align*}
 \langle f \rangle_{\alpha_q} :=   \sum_{\sigma \in \mathrm{L}_h} \frac{1}{|\Gamma_{\sigma i}|}\sum_{\substack{\tau \in T_q
 \\ }} \alpha_q(\tau \sigma) f(\tau \sigma),  \quad \alpha_q(\g):=\mathbf{1}_{\cc(\g)\equiv 0 \pmod{q}}.
 \end{align*}
Then for $k_{Z,1}:\G \to \C$ as in Theorem \ref{thm:technical}
\begin{align*}
  \sum_{\substack{q \sim Q \\ \gcd(h,q^2) \, \mathrm{ square-free} }}  \langle \alpha_q|\mathcal{K}_q k_{Z,1}  | \alpha_q\rangle \precprec Q h^{1/2} + h Z^{1/2}.
\end{align*}
    
\end{proposition}
\begin{proof}
By abuse of notation for $u >0$  we denote  $k_{Z,1}(u) = \mathbf{1}_{|u| \leq Z}(1+u)^{-1/2}$. Recall the point pair invariant $ u(w,z) = \frac{|w-z|^2}{4 \Im w \Im z}$ so that $u_1(\g) = u(\g i,i)$ and thus $k_{Z,1}(\g)=k_{Z,1}(u(\g i,i))$. We have by denoting $\g = \tau_2^{-1}\gamma \tau_1  \in \SL_2(\Z)$ and $\tau=\tau_2$
\begin{align*}
\langle\alpha_q|\mathcal{K}_q k_{Z,1}  | \alpha_q\rangle = & \sum_{\substack{\tau_1, \tau_2\in T_q\\ z_1, z_2\in \Lambda_h}} \sum_{\gamma \in \Gamma} 
 \alpha_q(\tau_1 z_1)  \alpha_q(\tau_2 z_2) k_{Z,1}\bigl(u(\gamma \tau_1 z_1, \tau_2 z_2)\bigr) \\
 =&  \sum_{\substack{\tau_1, \tau_2\in T_q\\ z_1, z_2\in \Lambda_h}} \sum_{\gamma \in \Gamma}  \alpha_q(\tau_1 z_1)  \alpha_q(\tau_2 z_2) k_{Z,1}\bigl(u(\tau_2^{-1}\gamma \tau_1 z_1,  z_2)\bigr) \\
 =& \sum_{z_1,z_2 \in \Lambda_h} \sum_{\g \in \SL_2(\Z)}  k_{Z,1}\bigl(u(\g z_1,  z_2)\bigr) \sum_{\tau \in T_q} \alpha_q(\tau \g z_1) \alpha_q(\tau z_2) \\
 =& \sum_{z_2 \in \Lambda_h} \sum_{w_1 \in \mathcal{S}_h}  k_{Z,1}\bigl(u(w_1,  z_2)\bigr) \sum_{\tau \in T_q} \alpha_q(\tau w_1) \alpha_q(\tau z_2),
\end{align*}
since $\g z_1=w_1$ ranges over the entire set $\mathcal{S}_h$. We shall below identify $w_1\cong(\aa_1,\bb_1,\cc_1)$ and $z_2\cong(\aa_2,\bb_2,\cc_2)$. 

\subsubsection{Diagonal}
The contribution from $w_1=z_2$ is bounded by
\begin{align*}
   \sum_{\substack{q \sim Q \\ \gcd(h,q^2) \, \mathrm{ square-free} }}   \sum_{z_2 \in \Lambda_h} k_{Z,1}\bigl(u(z_2,  z_2)\bigr)  \sum_{\tau \in T_q} \alpha_q(\tau z_2) \alpha_q(\tau z_2) \precprec Q h^{1/2},
\end{align*}
since for any $z_2 \in \Lambda_h$ we have $\sum_{\tau \in T_q} \alpha_q(\tau z_2)  \ll \gcd(\aa_2,\bb_2,\cc_2,q) d(q)$. Note that $\gcd(\aa_2,\bb_2,\cc_2,q) = 1$ for $\bb_2^2-\aa_2\cc_2=h$ by the assumption  that $\gcd(h,q^2)$ is square-free.
\subsubsection{Off-diagonal $w_1 \neq z_2$}
We choose a set of projective representatives  $(c_0,d_0) \in \mathbb{P}^1_q $ for $\tau = \smqty(\ast & \ast \\ c_0 & d_0) \in T_q$  and recall that for $z=(\aa,\bb,\cc)$ we have by \eqref{eq:cctransform}
\begin{align*}
    \cc(\tau z) = c_0^2 \aa + 2 c_0 d_0 \bb + d_0^2 \cc.
\end{align*}
Therefore, the sum over $\tau$ is bounded by $\ll d(q)$ times the indicator that for some $(c_0,d_0) \in \mathbb{P}^1_q $ we have
\begin{align} \label{eq:linearcondition}
    (c_0^2,2c_0 d_0,d_0^2) (\aa_j,\bb_j,\cc_j)^t \equiv 0 \pmod{q} \quad \text{for both $j\in\{1,2\}$.}
\end{align}
Denoting the cross-product 
\begin{align*}
    (\aa_1,\bb_1,\cc_1) \times   (\aa_2,\bb_2,\cc_2) = \mathbf{u}=(u_1,u_2,u_3),
\end{align*}
the condition \eqref{eq:linearcondition}  can happen only if $\mathbf{u} \equiv \lambda  (c_0^2,2c_0 d_0,d_0^2)$ for some $\lambda \in \Z/q\Z$ and  $(c_0,d_0) \in \mathbb{P}^1_q $. Define a quartic form in the six variables $\aa_j,\bb_j,\cc_j$ by $F(w_1,z_2) = u_2^2-2u_1 u_3$. Then our condition implies
\begin{align*}
    F(w_1,z_2)  \equiv 0 \pmod{q}.
\end{align*}
Importantly, for $w_1 \neq z_2$ we have that $F(w_1,z_2) \neq 0$  over $\R$. Indeed, the form being zero means that for some  $(c_0,d_0) \in \mathbb{P}^1_{\R}$ we have $\mathbf{u} \equiv \lambda  (c_0^2,2c_0 d_0,d_0^2)$ for some $\lambda \in \R$. Since $w_1 \neq z_2$ implies that $w_1$ and $z_2$ are not linearily dependent (by non-homogenuity of $\bb^2-\aa \cc = h$), we have $\lambda \neq 0$ and we would get
\begin{align*}
  (c_0^2,2c_0 d_0,d_0^2) (\aa_j,\bb_j,\cc_j)^t =0.
\end{align*}
 This is impossible, since by $\aa_j \cc_j -\bb_j^2= h > 0$ the above defines a  positive definite binary quadratic form in $(c_0,d_0)$. 
 
Therefore, for $w_1 \neq z_2$ we have
\begin{align*}
    \sum_{\tau \in T_q} \alpha_q(\tau w_1) \alpha_q(\tau z_2) \ll \gcd(\aa_2,\bb_2,\cc_2,q) d(q) \mathbf{1}_{0 \neq F(w_1,z_2) \equiv 0 \pmod{q}} \precprec  \mathbf{1}_{0 \neq F(w_1,z_2) \equiv 0 \pmod{q}},
\end{align*}
again using the assumption that $\gcd(h,q^2)$ is square-free. Since the kernel is supported on $|u|\leq Z$, we can restrict summation to
\begin{align*}
    u(w_1,z_2) =& \frac{|w_1-z_2|^2}{4 \Im w_1 \Im z_2} = \frac{(\frac{\bb_1}{\cc_1}- \frac{\bb_2}{\cc_2})^2 + h(\frac{1}{\cc_1} - \frac{1}{\cc_2} )^2}{ 4 h \frac{1}{\cc_1} \frac{1}{\cc_2}} \\
    =& \frac{\cc_1\cc_2}{4 h}  (\frac{\bb_1}{\cc_1}- \frac{\bb_2}{\cc_2})^2  + \frac{\cc_1\cc_2}{4} (\frac{1}{\cc_1} - \frac{1}{\cc_2} )^2 \leq Z. 
\end{align*}
Therefore, absorbing the sum over $q | F(w_1,z_2) \neq 0$ by a divisor bound and recalling \eqref{eq:Heegnerhull}, it suffices to count the number of solutions to
\begin{align*}
\mathbf{v}=&(\aa_1,\aa_2,\bb_1,\bb_2,\cc_1,\cc_2) \in \Z^6: \quad
 \bb_j^2+h=\aa_j \cc_j, \\
 &1 \leq \cc_2  \leq 2\sqrt{h}, \quad
    |\bb_2| <\cc_2, \quad
   \cc_1 \ll Z \cc_2, \quad
   |\bb_1| \ll Z\sqrt{h},
\end{align*}
weighted by $(1+u(w_1,z_1))^{-1/2}$. We consider three cases depending on which term dominates the weight $(1+u(w_1,z_1))^{-1/2}$. The contribution from $\cc_1 \leq 10\cc_2$ is bounded by $\precprec  h$. The contribution from  $\cc_1 > 10\cc_2$  and $|\bb_1|  \leq 10 \cc_1$ is bounded by
\begin{align*}
     \sum_{\cc_2 \ll \sqrt{h}} \sum_{\substack{ 0\leq \bb_2 \leq \cc_2 \\\bb_2^2+h \equiv 0 \pmod{\cc_2} } }  \sum_{10 \cc_2 < \cc_1 \ll Z \cc_2}  \sum_{\substack{ |\bb_1| \leq 10 \cc_1 \\\bb_1^2+h \equiv 0 \pmod{\cc_1} } } \frac{1}{1+ \cc_1/\sqrt{\cc_1 \cc_2}} \precprec      \sum_{\cc_2 \ll \sqrt{h}}  \sum_{10 \cc_2 < \cc_1 \ll Z \cc_2} \frac{\sqrt{\cc_2}}{\sqrt{\cc_1}} \precprec h Z^{1/2} .
\end{align*}
Finally, the contribution from $\cc_1 > 10\cc_2$  and $|\bb_1|  > 10 \cc_1$ 
\begin{align*}
 &\sum_{\cc_2 \ll \sqrt{h}} \sum_{\substack{ 0\leq \bb_2 \leq \cc_2 \\\bb_2^2+h \equiv 0 \pmod{\cc_2} } }  \sum_{\cc_1 \ll Z \cc_2} \sum_{10 \leq B \ll  \frac{Z \sqrt{h}}{\cc_1} } \sum_{\substack{ B \cc_1 \leq |\bb_1| \leq (B+1) \cc_1 \\\bb_1^2+h \equiv 0 \pmod{\cc_1} } } \frac{1}{1+ |\bb_1| \sqrt{\cc_2}/\sqrt{h\cc_1 }}  \\
&  \precprec  \sum_{\cc_2 \ll \sqrt{h}}  \sum_{\cc_1 \ll Z \cc_2} \sum_{ B \ll  \frac{Z \sqrt{h}}{\cc_1} } \frac{\sqrt{h}}{B \sqrt{\cc_1 \cc_2}}  \precprec h Z^{1/2}.
\end{align*}
\end{proof}

\subsection{Lower triangular orbits}
For Type II sums we require the following technical proposition. The appearing functional is motivated by Lemma \ref{lem:para2} with the identification $v=n_2 u_2-n_1u_1$ so that $u_1 \equiv - v \overline{n_1} \pmod{n_2}$. Since we sum over $\gamma \in \Gamma_0(sn_1n_2)$, it does not matter which representative for $\overline{n_1}$ we choose, as will be clear in the proof after the change of variables to $\g' =\smqty(a' & b'  \\ c' & d')$. 
\begin{proposition}  
\label{prop:lowertriang}
Let $D,N_0,N_1,N_2,T,V,Z \geq 1$ and $R > 0$.  Let $\beta_{s,n_1,n_2}$ be the linear functional defined by
 \begin{align*}
 \langle f \rangle_{\beta_{s,n_1,n_2}} :=   \mathbf{1}_{\gcd(n_1,n_2) = 1} \sum_{|v| \leq V}f(\n[-sn_1  v \,\overline{n_1} ]^t), 
 \end{align*}
 where $\overline{n_1} n_1 \equiv 1 \pmod{n_2}$. Then denoting  $s=dn_0t^2$ we have
 \begin{align*}
  \sum_{t \leq T} &\sum_{n_0 \leq N_0}  \sum_{n_1 \leq N_1} \sum_{n_2 \leq N_2} \sum_{\substack{d | n_0n_1 n_2 
 \\ d > D}}\langle  \beta_{s,n_1,n_2}|\mathcal{K}_{sn_1n_2} k_{Z,R} | \beta_{s,n_1,n_2} \rangle
    \\
      & \precprec  N_0 T V (1+R) (N_1 N_2 + N_1 V+N_2V)+ V^2 \bigl((DR)^{-1} + Z^{1/2}\bigr).
 \end{align*}
\end{proposition}
The reader may want to pretend in the first pass that $t=n_0=d=1$, which is the generic case as these variables arise from dealing with certain gcd conditions in the proof of Theorem \ref{thm:typeII}. In that case and assuming $N_1=N_2$, the bound simplifies to
\begin{align*}
    V^2(R^{-1} + Z^{1/2}) + (1+R) (V^2 N + V N^2).
\end{align*}
\begin{proof}
We want to estimate for $\gcd(n_1,n_2)=1$
\begin{align*}
    K(s,n_1,n_2) :&= \langle  \beta_{s,n_1,n_2}|\mathcal{K}_{sn_1n_2} k_{Z,R} | \beta_{s,n_1,n_2} \rangle\\
    &=\sum_{|v|,|v'| \leq V}\sum_{\gamma\in \Gamma }k_{Z,R}(\n[sn_1 v' \overline{n_1}]^t\gamma \n[-sn_1 v \overline{n_1}]^t),
\end{align*}
where $\Gamma=\Gamma_0(sn_1n_2)$ and 
\begin{align*}
    k_{Z,R}(\g) \ll  \frac{\mathbf{1}\{ |a| + |b|/R + R |c| + |d| \leq Z^{1/2} \}}{1+ |a| + |b|/R + R|c| + |d| }.
\end{align*}
We now rewrite the congruences coming from the group $\Gamma$. We have
\begin{align*}
\n[ s n_1 v' \overline{n_1}]^t\gamma \n[- s n_1 v \overline{n_1}]^t& =  \mqty(1 &\\s  n_1 v' \overline{n_1}  & 1) \mqty(a_0 & b_0 \\ c_0 s n_1 n_2 & d_0) \mqty(1 & \\ -sn_1 v\overline{n_1} & 1) = \mqty(a' & b' \\ c' & d')=:\g' \in \SL_2(\Z)
\end{align*}
with
\begin{align*}
   a'&=a_0-b_0sn_1 v \overline{n_1}, \quad \quad b'=b_0,  \quad  \quad d'=d_0+b_0 sn_1 v' \overline{n_1}, \\
   c'&= c_0 sn_1 n_2 -d_0 sn_1 v \overline{n_1}  + a_0 sn_1 v'\overline{n_1} - b_0 s^2n_1^2 vv' \overline{n_1}^2. 
\end{align*}
Thus, we deduce that $\g'$ and $v,v'$ satisfy the congruence conditions
\begin{align*}
    c' &\equiv 0 \pmod{sn_1}\\
v'a'-vd'- c'/s+s vv'b'  &\equiv 0  \pmod{n_2}
\end{align*}
Similar to the previous section, our goal is to absorb the moduli $sn_1$ and $n_2$ by a divisor bound, but we can only do so if the integers they divide are non-zero. We therefore split into four parts depending on whether $c' \neq 0$ or $c' = 0$ and $ v'a'-vd'- c'/s+s vv'b' \neq 0$ or $ v'a'-vd'- c'/s+s vv'b' = 0$, to get
\begin{align*}
    K(s,n_1,n_2) = \sum_{ \epsilon_1,\epsilon_2 \in \{=,\neq \} }   K_{\epsilon_1,\epsilon_2}(s,n_1,n_2).
\end{align*}

\subsubsection{Off-diagonal}
We separate $b'\neq0$ and $b'= 0$ to get $K_{\neq,\neq} =K_{\neq,\neq,\neq} + K_{\neq,\neq,=}$. 
For $b'\neq 0$ we have $b'c' \neq 0$ and $a'd' \neq 1$.  Using the divisor bound first for $n_2$ and then for $sn_1$, we have 
\begin{align*} 
K_{\neq,\neq,\neq}=\sum_{t \leq T} \sum_{n_0 \leq N_0}  \sum_{n_1 \leq N_1} \sum_{n_2 \leq N_2} \sum_{\substack{d | n_0n_1 n_2\\ d>D}} K_{\neq,\neq,\neq}(dn_0t^2,n_1,n_2)\precprec V^2 \sum_{\substack{a'd'-b'c'=1 \\ a'd' \neq 1}} k_{Z,R}(\g').   
\end{align*}
We can thus absorb the $b',c'$ by a divisor bound to get a contribution
\begin{align*}
   \precprec  V^2 \sum_{a',d'}  \frac{\mathbf{1}\{ |a'|  + |d'| \leq Z^{1/2}\}}{|a'| + |d'| } \precprec V^2 Z^{1/2} .
\end{align*}

For those terms with $b'=0$ we have $a'd' = 1$ and the congruences become
\begin{align*}
c'& \equiv 0 \pmod{sn_1} \\
  v'-v \pm c'/s  &\equiv 0  \pmod{ n_2 }.
\end{align*}
The part with $b'=0$ contributes
\begin{align*}
 K_{\neq,\neq,=} =   &\sum_{t \leq T} \sum_{n_0 \leq N_0}  \sum_{n_1 \leq N_1} \sum_{n_2 \leq N_2} \sum_{\substack{d | n_0n_1 n_2\\ d>D}}K_{\neq,\neq,=}(dn_0t^2,n_1,n_2)  \\
    =& \sum_{t \leq T} \sum_{n_0 \leq N_0}  \sum_{n_1 \leq N_1} \sum_{n_2 \leq N_2} \sum_{\substack{d | n_0n_1 n_2\\ d>D}}\sum_{|v|,|v'| \leq V}     \sum_{ \substack{0\neq c' \equiv 0 \pmod{s n_1} \\0\neq v'-v\pm c'/s  \equiv 0 \pmod{n_2} }}\frac{\mathbf{1}\{1 \leq |c'| \leq  Z^{1/2}/R\}}{|c'|R} \\
    \leq & \sum_{t \leq T} \sum_{n_0 \leq N_0}  \sum_{n_1 \leq N_1} \sum_{n_2 \leq N_2} \sum_{\substack{d | n_0n_1 n_2\\ d>D}}\sum_{|v|,|v'| \leq V}     \sum_{ \substack{0\neq c'' \equiv 0 \pmod{t^2 n_0 n_1} \\0\neq v'-v\pm c''(n_0t^2)  \equiv 0 \pmod{n_2} }}\frac{\mathbf{1}\{1 \leq |c''| \leq  Z^{1/2}/RD\}}{|c''|D R},
\end{align*}
where we have denoted $c'=d c''$ and used $d > D$. Using the divisor bound to absorb $d$, $n_2$, and then $n_0,n_1,t$, we get
\begin{align*}
     K_{\neq,\neq,=} \precprec  V^2  \sum_{\substack{c''}}\frac{\mathbf{1}\{1 \leq |c| \leq  Z^{1/2}/RD\}}{|c|D R} \precprec \frac{V^2}{D R}.
\end{align*}
\subsubsection{Pseudo-diagonal 1}
We note consider $c'=0$. Denoting $d_2=\gcd(n_2,d), d_1=d/d_2$, and by absorbing $n_2'=n_2/d_2$ by the divisor bound, we have
\begin{align*}
K_{=,\neq}&=\sum_{t \leq T} \sum_{n_0 \leq N_0}  \sum_{n_1 \leq N_1} \sum_{\substack{n_2 \leq N_2}} \sum_{d | n_0n_1 n_2} K_{=,\neq}(dn_0t^2,n_1,n_2)\\
&\precprec \sum_{t \leq T} \sum_{n_0 \leq N_0}  \sum_{n_1 \leq N_1} \sum_{\substack{d_2 \leq N_2}} \sum_{d_1 | n_0n_1} K_{=,\neq}(d_1 d_2n_0t^2,n_1,d_2).
\end{align*}
Since $c'=0$ and $d_2 | s$, the remaining congruence condition for $d_2$ is of the form
\begin{align*}
    v'a'-vd'\equiv 0 \pmod{d_2}.
\end{align*}
Again since $c'=0$, we have $a'd'=1$ and proceed by splitting up depending on whether $v'= v$ or not to get $K_{=,\neq,=} + K_{=,\neq,\neq}$.

We estimate the part with $v=v'$ crudely by
\begin{align*}
    K_{=,\neq,=} & \precprec \sum_{t \leq T} \sum_{n_0 \leq N_0}  \sum_{n_1 \leq N_1} \sum_{\substack{d_2 \leq N_2\\ \gcd(d_2,rn_0n_1)=1}} \sum_{d_1 | n_0n_1}  \sum_{\substack{|v|,|v'|\leq V\\ v'= v}}
\sum_{b'}  \frac{\mathbf{1}\{ |b'|  \leq R Z^{1/2}\}}{1+ |b'|/ R }   \\
   & \precprec   TN_0N_1N_2 V(1+R). 
   \end{align*}
For $v \neq v'$ we can absorb $d_2|(v'-v)$ by a divisor bound to get
\begin{align*}
      K_{=,\neq,\neq} & \prec\prec \sum_{t \leq T} \sum_{n_0 \leq N_0}  \sum_{n_1 \leq N_1}  \sum_{d_1 | n_0n_1}  \sum_{\substack{|v|,|v'|\leq V}}
\sum_{b'}  \frac{\mathbf{1}\{ |b'|  \leq R Z^{1/2}\}}{1+ |b'|/ R }  \precprec   TN_0N_1 V^2(1+R). 
\end{align*}
\subsubsection{Pseudo-diagonal 2}
We now consider $v'a'-vd'- c'/s+s vv'b' =0$. By denoting $d_1=\gcd(d,n_1), d_2=d/d_1$, and by absorbing $n_1/d_1$ by a divisor bound we get
\begin{align*}
 K_{\neq,=}   & = \sum_{t \leq T} \sum_{n_0 \leq N_0}  \sum_{n_1 \leq N_1} \sum_{\substack{n_2 \leq N_2\\ \gcd(n_2,tn_0n_1)=1}} \sum_{d | n_0n_1 n_2}  K_{\neq,=}(dn_0t^2,n_1,n_2)  \\
   & \precprec    \sum_{t \leq T} \sum_{n_0 \leq N_0}   \sum_{d_1 \leq N_1}\sum_{\substack{n_2 \leq N_2\\ \gcd(n_2,tn_0n_1)=1}} \sum_{d_2|n_0 n_2} K_{\neq,=}(d_1d_2n_0t^2,d_1,n_2).
\end{align*}
By definition we have
\begin{align*}
     K_{\neq,=}(s,d_1,n_2)  = \sum_{v,v' \leq V}\sum_{\substack{a'd'-b'c'=1 \\  sv'a'-svd'- c'+s^2vv'b' = 0  \\ 0 \neq c' \equiv 0 \pmod{sd_1} }} \frac{\mathbf{1}\{ |a'| + |b'|/R + R |c'| + |d'| \leq Z^{1/2}\}}{1+|a'| + |b'|/R + R|c'| + |d|' }.
\end{align*}
We substitute the equation $c'=sv'a'-svd'+ s^2vv'b'$ into the determinant equation to get
\begin{align*}
    a'd'-(sv'a'-svd'+ s^2vv'b')b'=1,
\end{align*}
which gives
\begin{align*}
    (a'+svb')(d'-sv'b') = 1.
\end{align*}
Then $a'=\pm 1 -svb', d'=\pm 1 + svb'$ and we get $c'=\pm s(v'-v) +  s^2 vv' b'$. Now, $sd_1|c'$ implies that $d_1|(v'-v)$. We again separate $v=v'$ and $v\neq v'$ to get $   K_{\neq,=,=} +    K_{\neq,=,\neq}$.

The part where $v=v'$ contributes
\begin{align*}
  K_{\neq,=,=} \precprec TN_0 N_1 N_2  \sum_{b'}  \frac{\mathbf{1}\{ |b'|  \leq R Z^{1/2}\}}{1+ |b'|/R }   \precprec TN_0N_1 N_2 V (1+R).
\end{align*}
For the part where  $v\neq v'$ we can absorb $d_1|(v-v')$ by a divisor bound to get
\begin{align*}
  K_{\neq,=,\neq} \precprec TN_0  N_2 V^2 \sum_{b'}  \frac{\mathbf{1}\{ |b'|  \leq R Z^{1/2}\}}{1+ |b'|/R }   \precprec TN_0N_2 V^2 (1+R).
\end{align*}

\subsubsection{Diagonal}
Lastly, we consider $c'=sv'a'-svd'- c'+s^2vv'b' = 0$. We have by the divisor bound for $d$ and crude estimates
\begin{align*}
  K_{=,=}  &= \sum_{t \leq T} \sum_{n_0 \leq N_0}  \sum_{n_1 \leq N_1} \sum_{\substack{n_2 \leq N_2\\ \gcd(n_2,tn_0n_1)=1}} \sum_{d | n_0n_1 n_2}  K_{=,=}(s,n_1,n_2)  \\
    &\precprec \sum_{t \leq T} \sum_{n_0 \leq N_0}  \sum_{n_1 \leq N_1} \sum_{\substack{n_2 \leq N_2}} \sum_{d | n_0n_1 n_2} \sum_{v,v' \leq V}\sum_{\substack{a'd'-b'c'=1 \\  sv'a'-svd'- c'+s^2vv'b' = 0  \\ c' = 0 }} \frac{\mathbf{1}\{ |a'| + R|b'| +  |c'|/R + |d'| \leq Z^{1/2}\}}{|a'| + R|b'|+ |c'|/R + |d|' } \\
   &  \precprec \sum_{t \leq T} \sum_{n_0 \leq N_0}  \sum_{n_1 \leq N_1} \sum_{\substack{n_2 \leq N_2}} \sum_{d | n_0n_1 n_2}\sum_{b'} \frac{\mathbf{1}\{    |b'|/R \leq Z^{1/2}\}}{1 + |b'|/R  }   \sum_{\substack{a'd'=1   }}\sum_{\substack{v,v' \leq V \\ sv'a'-svd'+ s^2vv'b' = 0 }}  1
 \\
 & \precprec T N_0 N_1 N_2  V \sum_{b'}  \frac{\mathbf{1}\{ |b'| R \leq Z^{1/2}\}}{1+ |b'| R } \precprec T N_0 N_1 N_2  V (1+R).
\end{align*}
\end{proof}

\section{Proof of Theorem \ref{thm:typeI}} \label{sec:typeIproof}
We may assume that for some small $\eta > 0$ we have $K  >X^{1-\eta}$, since otherwise the claim is trivial by applying Poisson summation on $\ell$.  Similarly, we may assume that $K \leq DX^{1+\eta}$ since otherwise the claim is trivial by switcing to the complementary divisor $m= \frac{a\ell^2+h}{k}$ and applying Poisson summation on $\ell$ modulo $dm$.

Let $X_1=X$ and $X_2 = K^{1/(1-\eta)} > X_1$.  Noting that the claim follows for $X=X_2$ trivially by the Poisson summation formula, it suffices to bound
\begin{align*}
     \sum_{d \leq D}\bigg| \sum_{k \equiv 0 \pmod{d}} \psi_1\Bigl(\frac{k}{K}\Bigr) \Bigl(\sum_{a \ell^2 +h \equiv 0 \pmod{k}} \psi_2\Bigl(\frac{\ell}{X}\Bigr)   -  \frac{X}{X_2}\psi_2\Bigl(\frac{\ell}{X_2}\Bigr)\Bigr) \bigg|.
\end{align*}
Normalising by $\sqrt{ah}$, we denote for real symmetric $\gf=\smqty(\aa & \bb \\ \bb & \cc )$ with determinant 1
\begin{align*}
    F_{\diamond,j}(\gf) =  \psi_1\Bigl(\frac{\cc \sqrt{a h} }{  a K}\Bigr) \frac{X}{X_j} \psi_2\Bigl(\frac{\bb \sqrt{a h}}{a X_j}\Bigr), 
\end{align*}
and define $F_j:\SL_2(\R) \to \C$ by $F_j(\g) = F_{\diamond,j}(\g \g^t)$, where $\aa= a_0^2+b_0^2, \bb= a_0 c_0+ b_0d_0, \cc=c_0^2+d_0^2$ for $\g = \smqty(a_0 & b_0 \\ c_0 & d_0)$. Then the Iwasawa coordinates are $y = \frac{1}{\cc}$ and $x= \frac{\bb}{\cc}$, and $F_{j}$ is supported on
\begin{align*}
\mathbf{X}_j& =   X^{o(1)}\frac{X_j}{K} \quad\quad \mathbf{Y} = X^{o(1)} \frac{h^{1/2}}{a^{1/2}}\frac{1}{K}, \quad \quad  \frac{\mathbf{X}_j}{\mathbf{Y}} = X^{o(1)} \frac{X_j a^{1/2}}{ h^{1/2}} = X^{o(1)} \frac{X_j }{ h^{1/2}} 
\end{align*}

Define the linear functional $\alpha=\alpha_{d,a,h}$ by
\begin{align}
    \langle f \rangle_\alpha=\sum_{\sigma \in \mathrm{L}_h} \frac{1}{|\Gamma_{\sigma i}|}\sum_{\substack{\tau \in T_q
 \\ }} \alpha_{d,a}(\tau \sigma) f(\tau \sigma),  \quad \alpha_{d,a}(\g):=\mathbf{1}_{\substack{\cc(\g)\equiv 0 \pmod{ad}\\ \bb(\g)\equiv 0 \pmod{a}}}. \label{eq:alphaTypeIdef}
\end{align}
We also let $I$ denote the linear functional $\langle f \rangle_I = f(I)$.

By Lemma \ref{lem:para} and by inserting the corresponding integrals over $\G$ which match exactly, we have
\begin{align*}
    \sum_{k \equiv 0 \pmod{d}} \psi_1\Bigl(\frac{k}{K}\Bigr) \Bigl(\sum_{a \ell^2 +h \equiv 0 \pmod{k}} \psi_2\Bigl(\frac{\ell}{X}\Bigr)   -   \frac{X}{X_2}\psi_2\Bigl(\frac{\ell}{X_2}\Bigr)\Bigr)\Bigr) = \langle I | \Delta_{ad} F_1 | \alpha_{d,a,h} \rangle  -  \langle I | \Delta_{ad} F_2 | \alpha_{d,a,h} \rangle.
\end{align*}
We can now apply Theorem \ref{thm:technical} to bound each of the terms. We only do the calculations for $j=1$, for the case $j=2$ we get better bounds thanks to the factor $X/X_2 < 1$.  Recalling that $h \leq X^{2+o(1)}$, we apply 
 Theorem \ref{thm:technical} with  $\delta^{-1} \precprec 1$,  $\Gamma=\Gamma_0(ad)$, $X h^{-1/2} = Z_0 Z_1 Z_2=Z$, and use Cauchy-Schwarz on $d$ to get
\begin{align*}
  \sum_{d\leq D} \Bigl|  \langle I | \Delta_{ad} F_1 | \alpha_{d,a,h} \rangle\Bigr|\precprec Z^{1/2}  Z_0^\theta \sqrt{K_1 K_2},
\end{align*}
where by positivity we have
\begin{align*}
    K_1&=\sum_{d\leq D}\langle  I |\Delta_{ad} k_{Z_1^2,\mathbf{X}} |I \rangle  \leq \sum_{q\leq aD} \langle I |\mathcal{K}_{q} k_{Z_1^2,\mathbf{X}}| I \rangle.   \\
    K_2&=\sum_{d\leq D} \langle \alpha_{d,a,h}| \Delta_{ad} k_{Z_2^2,1} | \alpha_{d,a,h} \rangle  \leq \sum_{q\leq aD} \langle \alpha_{q}| \mathcal{K}_{q} k_{Z_2^2,1} | \alpha_{q} \rangle,
\end{align*}
where the functionals $\alpha_q$ are as in Proposition \ref{prop:heegner}. By Proposition \ref{prop:lowertriang} with $q=n_1$ and $D=N_0=N_2=T=V=1$, we get 
\begin{align} \label{eq:K1bound}
  K_1  \precprec D(1+\mathbf{X})+\mathbf{X}^{-1}+Z_1.
\end{align}
By Proposition \ref{prop:heegner} we get
\begin{align}\label{eq:K2bound}
     K_2   \precprec D h^{1/2} + h Z_2.
\end{align}
The result follows by choosing $Z_1 = D $ and $Z_2 = 1+D h^{-1/2}$ so that $Z_0 = 1+X^{o(1)} \frac{X}{D(h^{1/2}+D)}$. Recall that at the beginning we reduced to the case $K \leq DX^{1+\eta}$ so that $K_1 \precprec D+ Z_1$ once we let $\eta = o(1)$.
\qed

\section{Proof of Theorem \ref{thm:typeII}}
In the first pass the reader may wish to focus on the generic case where $t=n_0=d=1$ below. We may assume that for some small $\eta > 0$ we have $MN >X^{1-\eta}$, since otherwise the claim is trivial by applying Poisson summation on $\ell$. We denote $X_1=X$, $X_2= (MN)^{1/(1-\eta)}$,  $\psi_{j}(u)  = \frac{X}{X_j} \psi(u/X_j)$ for $j=1,2$ and $\Psi = \psi_1-\psi_2$. Note again that the claim is trivial for $X=X_2$ by Poisson summation on $\ell$. Denoting 
$t=\gcd(m,n)$ and making substitutions $(m,n)\mapsto (tm,tn)$, it suffices to bound
\begin{align*}
    A(M,N):= &\sum_{t\leq 2N}\sum_{\substack{m\sim M/t }} \sum_{\substack{n\sim N/t \\ \gcd(n,mt)=1}}  \alpha_{mt} \beta_{nt}  \sum_{a \ell^2 +h \equiv 0 \pmod{mnt^2}} \Psi(\ell).
\end{align*}
We now wish to apply Cauchy-Schwarz analogously to \cite[Section 3.5]{mlargepf}.  Denoting
\begin{align*}
    Q=Q_{m,t}= \prod_{\substack{p \leq 2N 
\\ p\,\nmid \,mt \\ \varrho_{a,h}(p) \neq 0}} p
\end{align*} and using the fact that $\beta_n$ are supported on square-free integers, we have by the Chinese remainder theorem
\begin{align*}
 A(M,N)    =&  \sum_{\substack{t \leq 2N\\ m\sim M/t }} \alpha_{mt}  \frac{1}{\varrho_{a,h}(Q)}\sum_{\substack{\ell_0 \pmod{m t^2 Q} \\ a\ell_0^2 +h \equiv 0 \pmod{m t^2 Q} }}  \sum_{\substack{n \sim N/t\\ \gcd(n,mt)=1}} \beta_{nt} \varrho_{a,h}(n)  \sum_{\ell\equiv \ell_0 \pmod{mnt^2}}  \Psi(\ell).
\end{align*}
Applying Cauchy-Schwarz on $t,m,\ell_0$ we get
\begin{align} \label{eq:cstypeii}
    A(M,N)    \precprec M^{1/2} B(M,N)^{1/2},
\end{align}
where, by the same argument as in \cite[Section 3.5]{mlargepf} and another application of the Chinese remainder theorem,
\begin{align*}
   B(M,N) =    &\sum_{t \leq 2N}\sum_{n_0 \leq 2N}\sum_{\substack{n_1,n_2 \sim N/tn_0\\ \gcd(n_1,n_2)=1}} \beta_{n_0 n_1 t } \overline{\beta_{n_0 n_2 t} }\varrho_{a,h}(n_0)    \\
&\times\sum_{\gcd(m,n_0n_1n_2)=1} \psi\big(\frac{mt}{M}\big) \sum_{\substack{a \ell_j^2 +h \equiv 0 \pmod{mn_0n_jt^2} \\ \ell_1 \equiv \ell_2 \pmod{ m n_0 t^2 } }} \Psi(\ell_1)\Psi(\ell_2).
\end{align*}
It then suffices to show that
\begin{align} \label{eq:Bmnclaim}
    B(M,N) \precprec X +\frac{X N}{M} (N+h^{1/4}) \bigg( 1+ \frac{X^2}{MN^2(N^2 + h^{1/2})} \bigg)^{\theta}.
\end{align}
We write $B= B_{=} + B_{\neq}$ to separate the diagonal $\ell_1=\ell_2$ and the off-diagonal $\ell_1 \neq \ell_2$.
\subsection{Diagonal contribution $\ell_1=\ell_2$}
By absorbing $t,n_0,n_1,n_2,m$ with a divisor bound, we get
\begin{align*}
B_{=}(M,N) 
\precprec \sum_{\ell} \Psi(\ell)^2 \precprec X.
\end{align*} 
\subsection{Off-diagonal contribution}
By expanding the coprimality condition for $m$ with the Möbius function, we have
\begin{align*}
B_{\neq}(M,N)  =    &\sum_{tn_0 \leq 2N}\sum_{\substack{n_1,n_2 \sim N/tn_0 \\ \gcd(n_1,n_2)=1  }} \beta_{n_0 n_1 t } \overline{\beta_{n_0 n_2 t} }\varrho_{a,h}(n_0) \sum_{\substack{d  \\ d | n_0 n_1 n_2 } } \mu(d)   \\
   & \times\sum_{m} \psi(\frac{mdt}{M }) \sum_{\substack{a \ell_j^2 +h \equiv 0 \pmod{m d n_0n_jt^2} \\ \ell_1 \equiv \ell_2 \pmod{ m d n_0 t^2} \\ \ell_1 \neq \ell_2   }} \Psi(\ell_1)\Psi(\ell_2).
\end{align*}
Applying Lemma \ref{lem:para2} we get for $s=d n_0t^2$
\begin{align*}
    \sum_{m} \psi(\frac{mdt}{M}) \sum_{\substack{a \ell_j^2 +h \equiv 0 \pmod{s mn_j} \\ \ell_1 \equiv \ell_2 \pmod{ m s  } }} \Psi(\ell_1)\Psi(\ell_2) = \sum_{(u_1,u_2) \in U} \sum_{\gf \in \mathrm{S}_{a,h}(sn_1n_2)}f_2(\n[asn_1 u_1]^t \diamond \gf, \n[asn_2 u_2]^t \diamond \gf ), 
\end{align*}
where
\begin{align*}
    f_2(\gf_1,\gf_2) &= \Psi(\bb_1/a) \Psi(\bb_2/a) \psi(\frac{\aa_1 dt}{M}) = \sum_{i,j \in \{1,2\} } (-1)^{i+j} f_{ij}(\gf_1,\gf_2),\\
   f_{ij}(\gf_1,\gf_2) &=  \psi_i(\bb_1/a) \psi_j(\bb_2/a) \psi(\frac{\aa_1 dt}{M})
\end{align*}

We now transform this sum for application of Theorem \ref{thm:technical}. Let $v=  n_2 u_2- n_1 u_1$. The first step is to use Fourier inversion to relax a smooth cross-condition between $v$ and $\gf$. Before this we need to truncate the sum over $v$. We claim that it is supported on
\begin{align}\label{eq:vtrunc}
|v| \leq V_{ij}:= 10 \frac{\max\{X_i,X_j\}}{Mn_0t}.  
\end{align}
Indeed, we have
\begin{align*}
    \gf_2 = \n[as n_2 u_2- as n_1 u_1]^t \diamond \gf_1
\end{align*}
where $\n[x]^t$ maps $\bb_1 \mapsto \bb_1 + x \aa_1.$ Then, for $i\leq j$ and $\aa_1=\aa_2$, we have $F_{ij}(\gf_1,\gf_2)=0$ for $|x| > 10 \frac{dt X_j}{M}$ and vice versa for $i \geq j$. The truncation \eqref{eq:vtrunc} follows. We apply Fourier inversion to get
\begin{align*}
    f_{ij}(\n[a sn_1 u_1]^t \diamond \gf, \n[a sn_2 u_2]^t \diamond \gf ) &= \int_\R   \frac{X^2}{X_iX_j} \frac{V_{ij}}{(1+|\xi| V_{ij})^{2}} f_{ij,\xi}(\n[a sn_1 u_1]^t \diamond \gf) e(asv\xi   ) \d \xi ,
\end{align*}
where
\begin{align} \label{eq:fourier}
     f_{ij,\xi}(\gf) := \frac{X_iX_j}{X^2}\frac{(1+|\xi| V_{ij})^{2}}{V_{ij}}\int_\R  f_{ij}(\gf, \n[u]^t \diamond  \gf) e(-\xi u)\d u
\end{align}
satisfies by integration by parts uniformly in $\xi$
\begin{align*}
    | f_{ij,\xi}(\gf)| \ll 1.
\end{align*}
Observe that $u_1 \equiv -v \overline{n_1} \pmod{n_2}$. The choice of representative for $\overline{n_1}$ does not matter once we sum over $\gf \in \mathrm{S}_{a,h}(sn_1n_2)$. By $\ell_1 \neq \ell_2 $ we have $v\ \neq 0$. Thus, we arrive at
\begin{align*}
   & \sum_{m} \psi(\frac{mdr}{M}) \sum_{\substack{a \ell_j^2 +h \equiv 0 \pmod{m d n_0n_jr^2} \\ \ell_1 \equiv \ell_2 \pmod{ m d n_0 r^2} \\ \ell_1 \neq \ell_2  }} \Psi(\ell_1)\Psi(\ell_2)  \\
    =&   \int_{\R} \sum_{i,j \in \{1,2\} } (-1)^{i+j} \frac{X^2}{X_iX_j}\frac{V_{ij}}{(1+|\xi| V_{ij})^{2}}\Bigg(\sum_{1\leq |v| \leq V_{ij}} e(asv \xi) \sum_{\gf \in \mathrm{S}_{a,h}(sn_1n_2)}f_{ij,\xi}(\n[-asn_1 v \overline{n_1}]^t \diamond \gf) \Bigg) \d \xi.
\end{align*}
This has brought our counting problem into the right shape for an application of Theorem \ref{thm:technical} and we next show that the smooth weight is admissible for this. We normalise the weight by $\sqrt{ah}$ to $ \SL_2(\R)$ by defining
\begin{align*}
    F_{ij,\xi}(\g) = f_{ij,\xi}(\sqrt{ah} \, \g \g^t),
\end{align*}
which is supported on $\g\g^t = \smqty(\aa & \bb \\ \bb & \cc)$ with $\aa \cc - \bb^2 =1$ satisfying
\begin{align*}
   x&= \frac{\bb}{\cc} = \frac{\aa \bb}{\bb^2+  1} = X^{o(1)} \frac{\aa}{\bb} = X^{o(1)} \frac{ M}{a dt X_i} =  \mathbf{X}_i \\
  y &= \frac{1}{\cc}= \frac{\aa}{\bb^2+1} = X^{o(1)} \frac{\aa}{\bb^2} = X^{o(1)}\frac{\sqrt{h} M}{a^{3/2} dt X^{2}_i} = \mathbf{Y}_i.
\end{align*}
 Then
\begin{align*}
 \mathbf{X}_i /\mathbf{Y}_i  = X^{o(1)} X_i a^{1/2} h^{-1/2}  =X^{o(1)} X_i h^{-1/2} .
\end{align*}
The derivative hypothesis in Theorem \ref{thm:technical} follows from differentiating under the integration in \eqref{eq:fourier}, and is uniform in $\xi$. Note also that it suffices to bound the supremum over $\xi$ since $\int_{\R} \frac{V_{ij}}{(1+|\xi| V_{ij})^{2}}\d \xi \ll 1$.

We define the linear functionals  $\alpha=\alpha_{sn_1n_2,a,h}$ (being very similar to the one in the proof of Proposition \ref{thm:typeI}) by
\begin{align*}
    \langle f \rangle_\alpha=\sum_{\sigma \in \mathrm{L}_h} \frac{1}{|\Gamma_{\sigma i}|}\sum_{\substack{\tau \in T_q
 \\ }} \alpha_{sn_1n_2,a}(\tau \sigma) f(\tau \sigma),  \quad \alpha_{sn_1n_2,a}(\g):=\mathbf{1}_{\substack{\cc(\g)\equiv 0 \pmod{sn_1n_2a}\\ \bb(\g)\equiv 0 \pmod{a}}}.
\end{align*}
We define the linear functionals $\beta=\beta_{as,n_1,n_2,\xi}$ by
\begin{align*}
    \langle f\rangle_{\beta}=\sum_{1 \leq |v|\leq V_{ij}} e(asv \xi) f(\n[asn_1v\overline{n_1}]^t).
\end{align*}
By the parametrisation in Lemma \ref{lem:para}, we then have
\begin{align*}
    \sum_{|v| \leq V_{ij}} e(asv \xi) \sum_{\gf \in \mathrm{S}_{a,h}(sn_1n_2)}f_{ij,\xi}(\n[-a sn_1 v \overline{n_1}]^t \diamond \gf) = \langle  \beta |\mathcal{K}_{as n_1 n_2 } F_{ij}| \alpha \rangle.
\end{align*}
We can now apply Theorem \ref{thm:technical} with $\Gamma=\Gamma_0(a s n_1n_2)$. The main terms $\int_\G (\cdots) \d \g$ match exactly for $i,j \in \{1,2\}$ and it suffices to bound $\langle  \beta |\Delta_{as n_1 n_2 } F_{ij}| \alpha \rangle$ separately for each $i,j \in \{1,2\}$. The error term for $i=j=1$  dominates (thanks to the factor $\frac{X^2}{X_iX_j}$), so we restrict to that case. Together with an application of Cauchy-Schwarz on the $t,n_0,d$ variables, we get for $Z_0Z_1 Z_2 = \mathbf{X}_1/\mathbf{Y}_1$
\begin{align} \label{eq:IIapplication}
  \sum_{t n_0 \leq 2N}   \sum_{\substack{n_1,n_2 \sim N/tn_0 \\ \gcd(n_1,n_2)=1 }} \sum_{d | n_0n_1n_2} \Bigl|\langle \beta| \Delta_{a s n_1 n_2 } F_{11} | \alpha \rangle\Bigr| \precprec_a (\mathbf{X}_1/\mathbf{Y}_1)^{1/2} Z_0^\theta \sqrt{K_1 K_2}.
\end{align}
We estimate $K_1$ by Proposition \ref{prop:lowertriang} with $Z_1=\frac{M N^2}{X}$. This gives
\begin{align*}
    K_1\leq &\sum_{t n_0 \leq 2N} \sum_{\substack{n_1,n_2 \sim N/tn_0 \\ \gcd(n_1,n_2)=1 }} \sum_{d | n_0n_1n_2}  \langle  \beta_{s,n_1,n_2}|\mathcal{K}_{sn_1n_2} k_{Z_1^2,\mathbf{X}} |\beta_{s,n_1,n_2}\rangle\\
 \precprec&  \frac{X}{M} N^2+\frac{X^2}{M^2} Z_1 \precprec \frac{X}{M} N^2.
\end{align*}
For $K_2$ we apply Proposition \ref{prop:heegner} with $Z_2 = 1+ N^2 h^{-1/2}$, getting
by the divisor bound
\begin{align*}
    K_2 \leq&   \sum_{t n_0 \leq 2N}   \sum_{\substack{n_1,n_2 \sim N/tn_0 \\ \gcd(n_1,n_2)=1 }} \sum_{d | n_0n_1n_2}  \langle \alpha_{sn_1n_2,a,h}|\mathcal{K}_{asn_1n_2} k_{Z_2,1} |\alpha_{sn_1n_2,a,h}\rangle \\
    \precprec& \sum_{q \leq a N^2}  \langle \alpha_{q,h}| \mathcal{K}_{q} k_{Z_2,1} | \alpha_{q,h} \rangle 
    \precprec N^2 h^{1/2} + h Z_2 \precprec  N^2 h^{1/2} + h.
\end{align*}
Plugging the bounds for $K_1$ and $K_2$ into \eqref{eq:IIapplication} with $ \mathbf{X}_1 /\mathbf{Y}_1 =X^{1+o(1)} h^{-1/2} $ and
\begin{align*}
    Z_0 =1+ X^{o(1)}\frac{X}{h^{1/2} 
 Z_1Z_2} = 1+X^{o(1)}\frac{X^2}{MN^2(N^2 + h^{1/2})},
\end{align*}
we get \eqref{eq:Bmnclaim}. This completes the proof of Theorem \ref{thm:typeII}. \qed

\section{Proof of Theorems \ref{thm:primefactor} and \ref{thm:roots}}

We now state two corollaries to Theorems \ref{thm:typeI} and \ref{thm:typeII} that give explicit Type I and Type II ranges for $h \leq X^{1+\eps}$. We get power saving in the full range as soon as there is a spectral gap, that is, if $\theta < 1/2$. This makes our main theorems independent of progress towards the Selberg eigenvalue conjecture.

\begin{corollary}[Explicit Type I information] 
Let $\eta,\eps > 0$ be small, Let $ K \leq X^{2}$ and $D \leq X^{1/2-\eta}$. Let $h$ be square-free with $h \leq X^{1+\eps}$ and let $a \leq X^{o(1)}$ with $\gcd(a,h)=1$. Let $\psi_1,\psi_2$ be smooth functions supported on $[1,2]$ and $[-1,1]$, resp., such that for all $J \geq 0$ we have $\psi_i^{(J)} \precprec_J 1$. Then
\begin{align*}
    \sum_{d \leq D}\bigg| \sum_{k \equiv 0 \pmod{d}} \psi_1\Bigl(\frac{k}{K}\Bigr) \Bigl(\sum_{a \ell^2 +h \equiv 0 \pmod{k}} \psi_2\Bigl(\frac{\ell}{X}\Bigr)   - \frac{\varrho_{a,h}(k) X \widehat{\psi_2}(0) }{k} \Bigr) \bigg| & \precprec X^{1-(1-2\theta)\eta + \eps/4}
\end{align*}
\end{corollary}

\begin{corollary}[Explicit Type II information]  
Let $\eta,\eps> 0$ be small  and let $M \geq N \geq 1$ with $MN=X^\alpha$ satisfy
\begin{align*}
    X^{\alpha-1+2\eta} \leq N \leq X^{(2-\alpha)/3 - \frac{4}{3}\eta }.
\end{align*}
 Let $h$ be square-free with $1 \leq h \leq X^{1+\eps}$ and let $a \leq X^{o(1)}$ with $\gcd(a,h)=1$. Let $\psi$ be a smooth function supported  $[-1,1]$,  such that for all $J \geq 0$ we have $\psi^{(J)} \precprec_J 1$. Suppose that $\alpha_m,\beta_n$ are divisor bounded coefficients with $\beta_n$ is supported on square-free integers. Then
\begin{align*}
    &\sum_{\substack{m\sim M \\ n\sim N}} \alpha_m \beta_n  \Bigl(\sum_{a \ell^2 +h \equiv 0 \pmod{mn}} \psi\Bigl(\frac{\ell}{X}\Bigr)   - \frac{\varrho_{a,h}(mn) X \widehat{\psi}(0) }{mn} \Bigr)  \precprec   X^{1-(1-2\theta) \eta + \eps/8}.
\end{align*}
\end{corollary}

 We obtain precisely the Type I and Type II ranges that were obtained in \cite{mlargepf} under the assumption of Selberg's eigenvalue conjecture. By the calculations in \cite[proof of Theorem 2]{mlargepf}, Theorem \ref{thm:primefactor} follows. The assumption on $\varrho_{a,h}(p)$ ensures that the sieve dimension is $\leq 1 +\eps$ in the application of the linear sieve and Harman's sieve in \cite{mlargepf}.

Similarly, by exactly the same sieve argument as in \cite{DFIprimes}, these imply Theorem \ref{thm:roots}. The sieve argument in \cite{DFIprimes} works unchanged as there a two-dimensional sieve is used to bound the small contribution from the discarded ranges.

\section{A divisor problem for $ax^2+by^3$} \label{sec:x2y3}

We recall the set-up in \cite{mx2y3}. Let $a,b>0$ be coprime integers and let $\delta = \delta(X):= (\log X)^{-c}$ for some fixed large $c >0$. Let $f,f_1,f_2$ denote non-negative non-zero smooth functions supported in $[1,1+\delta]$ and satisfying the derivative bounds $ f^{(J)}, f_1^{(J)}, f_2^{(J)}\ll_J \delta^{-J}$ for all $J \geq 0$.  For  $A \in (\delta X^{1/2}, X^{1/2}]$ and $B\in (\delta X^{1/3}, X^{1/3}]$ we define the sequences $\mathcal{A}=(a_n)$, $\mathcal{B}=(b_n)$, and their difference $\mathcal{W}=(w_n)$ by 
\begin{align*}
    a_n &:= a_n(a,b,f_1,f_2,A,B) = \sum_{n=a x^2+b y^3} f_1(\tfrac{x}{A}) f_2(\tfrac{y}{B}), \\
    b_n &:=  b_n(a,b,f,f_1,f_2,A,B,W)=  f(\tfrac{n}{X}) \frac{AB \widehat{f_1}(0)\widehat{f_2}(0)  }{X\widehat{f}(0)  } \\
    w_n &:= a_n-b_n
\end{align*}

Then we have the following proposition, which is a dyadic version of the divisor problem $d(ax^2+by^3)$ along multiples of a modulus $d$ of size up to $X^{1/4-o(1)}$.
\begin{proposition}[Type I$_2$ information up to $1/4$] \label{prop:typeI2}
 Let $a,b > 0 $ be coprime integers and let $\eta > 0$ be sufficiently small. For any $K \leq X^{3/4}$  we have
  \begin{align*}
    \sum_{d \leq X^{1/4-\eta}} \bigg| \sum_{k \equiv 0\pmod{d} } \mu^2(k) f(\tfrac{k}{K}) 
 \sum_{n }  w_{k n} \bigg| \ll X^{5/6-\eta^3}.
  \end{align*}
\end{proposition}
We first reduce the proof to the following technical version, which restricts to a good set of $y$ at the cost of slightly increasing $b$. For the good set of $y$ we show a version with stronger power saving.
\begin{proposition}\label{prop:typeI2tech}
Denote $\varrho_{a,by^3}(d) := \# \{ x \in \Z/d\Z: ax^2+by^3 \equiv 0 \pmod{d} \}$.
Let $b \leq A^{O(\eta^2)}$, $D < B^{1-\eta}$, and $B^3 \leq A^{2+O(\eta^2)}$. Then we have
\begin{align*}
     \sum_{d \sim D}     \sum_{\substack{y \sim B \\ \gcd(y,ab)=1}} \mu^2(y) &\bigg|\sum_{\substack{k \equiv 0 \pmod{d}  }} f(\tfrac{k}{K})\bigg(  \sum_{\substack{ ax^2+by^3 \equiv 0 \pmod{k}}}   f_1(\tfrac{x}{A}) -\frac{A \widehat{f_1}(0) \varrho_{a,by^3}(d) }{d} \bigg) \bigg| \\ &\ll A^{O(\eta^2)} \bigg(1+ \frac{A}{B D (B^{1/2}+ D)} \bigg)^\theta \bigg(  A^{1/2} B D+ A^{1/2}  B^{5/4} D^{1/2}
 \bigg).
\end{align*}
\end{proposition}
We need the following lemma to show that the main terms match.
\begin{lemma} \label{le:ypoisson}
    For $d \leq B^{1-\eta}$ square-free and $a,b > 0$ co-prime we have
    \begin{align*}
        \sum_{y} f_2(\tfrac{y}{B})\varrho_{a,by^3}(d) &= \frac{B \widehat{f_2}(0)}{d} \sum_{y \pmod{d}} \varrho_{a,by^3}(d) + O(B^{\eta^2} d^{1/2+o(1)}) \\
       & = B \widehat{f_2}(0) + O(B^{\eta^2} d^{1/2+o(1)}). 
    \end{align*}
 \end{lemma}
\begin{proof}
    By the truncated Poisson summation formula (cf. \cite[Section 5]{mx2y3}) we have
    \begin{align*}
         \sum_{y} f_2(\tfrac{y}{B})\varrho_{a,by^3}(d) = \frac{B \widehat{f_2}(0)}{d} \sum_{y \pmod{d}} \varrho_{a,b,y}(d) + O(B^{\eta^2}  E  ) + O_\eta( X^{-100}),
    \end{align*}
    where for $H= X^{\eta^2} d/B$ we get by the Weil bound (cf. \cite[Section 5]{mx2y3})
    \begin{align*}
        E  &= \frac{1}{H} \sum_{\substack{0< |h|  \leq H} }\bigg| \sum_{\substack{x,y \pmod{d} \\ ax^2+by^3 \equiv 0 \pmod{d} }} e_d( h y) \bigg| \ll  d^{1/2+o(1)}. \qedhere
    \end{align*}
\end{proof}
\begin{proof}[Proof of Proposition \ref{prop:typeI2} assuming Proposition \ref{prop:typeI2tech}]
With an application of Lemma \ref{le:ypoisson} and expanding the $\mu(k)^2$ via M\"obius function, trivially bounding the contribution from square-divisors larger than $X^{\eta^2}$, it suffices to show that
 \begin{align*}
      \sum_{d \sim D}    \sum_{y \leq B} &\bigg|\sum_{\substack{k \equiv 0 \pmod{d}  }} f(\tfrac{k}{K})\bigg(  \sum_{\substack{ ax^2+by^3 \equiv 0 \pmod{k}}}   f_1(\tfrac{x}{A}) -\frac{A \widehat{f_1}(0) \varrho_{a,by^3}(d) }{d} \bigg) \bigg| \ll_{a,b} X^{5/6-2\eta^3}.
 \end{align*}
We can divide the congruence throughout by $\gcd(a,y^3)$, and modify $a,b,B,D,K$ accordingly. It then suffices to show the same bound for
\begin{align*}
     \sum_{d \sim D}    \sum_{\substack{y \leq B \\ \gcd(y,a)=1}} &\bigg|\sum_{\substack{k \equiv 0 \pmod{d}  }} f(\tfrac{k}{K})\bigg(  \sum_{\substack{ ax^2+by^3 \equiv 0 \pmod{k}}}   f_1(\tfrac{x}{A}) -\frac{A \widehat{f_1}(0) \varrho_{a,by^3}(d) }{d} \bigg) \bigg|.
\end{align*}
We split $y=y_0y_1y_2$, where $y_2|y$ is the largest square-free divisor such that $\gcd(y_2,bdy_0y_1)=1$, $y_0|(bd)^\infty$, and $y_1$ is powerful. Then the above is bounded by
\begin{align*}
     \sum_{d \sim D}     \sum_{\substack{y_0 y_1 \leq B \\ \gcd(y_0 y_1,a)=1 \\ y_0 | (bd)^\infty \\ y_1 \, \text{powerful} }} &\sum_{\substack{y_2 \leq B/y_0y_1 \\ \gcd(y_2,abdy_0y_1)=1}} \mu^2(y_2)  \\
     & \times \bigg|\sum_{\substack{k \equiv 0 \pmod{d}  }} f(\tfrac{k}{K})\bigg(  \sum_{\substack{ ax^2+by_0^3 y_1^3 y_2^3 \equiv 0 \pmod{k}}}   f_1(\tfrac{x}{A}) -\frac{A \widehat{f_1}(0) \varrho_{a,by_0^3 y_1^3 y_2^3 }(d) }{d} \bigg) \bigg|.
\end{align*}
The contribution from the part where $y_0y_1 > X^{\eta^2}$ is small by crude bounds, and for $y_0y_1 \leq X^\eta$ we may absorb $y_0^3y_1^3$ into $b$ and apply Proposition \ref{prop:typeI2tech} to get the claim.
\end{proof}

\subsection{Proof of Proposition \ref{prop:typeI2tech}}

Denote for real symmetric $\gf=\smqty(\aa & \bb \\ \bb & \cc )$ with determinant 1
\begin{align*}
    F_{\diamond,1}(\gf) =  f\Bigl(\frac{\cc \sqrt{a by^3} }{  a K}\Bigr)  f_1\Bigl(\frac{\bb \sqrt{a by^3}}{a A}\Bigr) 
\end{align*}
and define $F_1:\SL_2(\R) \to \C$ by $F_1(\g) = F_{\diamond,1}(\g \g^t)$, where $\aa= a_0^2+b_0^2, \bb= a_0 c_0+ b_0d_0, \cc=c_0^2+d_0^2$. Then the Iwasawa coordinates are $y = \frac{1}{\cc}$ and $x= \frac{\bb}{\cc}$, and we have
\begin{align*}
\mathbf{X}& \asymp  \frac{A}{K} \quad\quad \mathbf{Y} \asymp  \sqrt{\frac{b B^3}{a}}\frac{1}{K}.
\end{align*}
We  may assume that $K > A^{1-\eta^2}$ as otherwise the claim follows quickly by Poisson summation on $x$. Then $\mathbf{X} \leq A^{\eta^2}$ and $\mathbf{X}/\mathbf{Y}  > A^{-\eta^2}$, verifying hypotheses of Theorem \ref{thm:technical} with $\delta = X^{-\eta^2}$.

Recall the definition \eqref{eq:alphaTypeIdef} for a linear functional $\alpha=\alpha_{d,a,h}$, where we now set $h=by$. We also let $I$ denote the linear functional $\langle f \rangle_I = f(I)$. Then by Lemma \ref{lem:paracube} with $h=by$ for $y$ square-free with $\gcd(y,ab)=1$, recalling the normalization  $y^{-1/2}$ in the definition of the Hecke operator, we have (defining $F_2$ analogously to Section \ref{sec:typeIproof} with $A_2= K^{1/(1-\eta^2)}$)
\begin{align*}
   \sum_{\substack{k \equiv 0 \pmod{d}  }} &f(\tfrac{k}{K})\bigg(  \sum_{\substack{ ax^2+by^3 \equiv 0 \pmod{k}}}   f_1(\tfrac{x}{A}) -\frac{A \widehat{f_1}(0) \varrho_{a,by^3}(d) }{d} \bigg)  \\
   &= y^{1/2}\langle I|\mathcal{T}_{y,1} \Delta_{da} F_1| \alpha_{d,a,by}\rangle  - y^{1/2}\langle I|\mathcal{T}_{y,1} \Delta_{da} F_2| \alpha_{d,a,by}\rangle+O(X^{-100}).
\end{align*}
The error term from $F_1$ again dominates. We can now apply Theorem \ref{thm:technical}  with $\Gamma=\Gamma_0(ad)$, $A B^{-3/2}=Z = Z_0 Z_1 Z_2$, and Cauchy-Schwarz on $d$ to get
\begin{align*}
  \sum_{d \sim D}     \sum_{\substack{y \sim B \\ \gcd(y,ab)=1}} \mu^2(y) y^{1/2} \Bigl| \langle I|\mathcal{T}_{y,1} \Delta_{da} F_1| \alpha_{d,a,by}\rangle \Bigr|\ll X^{O(\eta^2)} B Z^{1/2} Z_0^\theta \sqrt{K_1 K_2},
\end{align*}
where by \eqref{eq:K1bound} 
\begin{align*}
      K_1  \ll \mathbf{X}^{-1}+Z_1+D(1+\mathbf{X})
\end{align*}
and by \eqref{eq:K2bound}, bounding trivially the Heegner points associated to $b \leq X^{\eta^2}$, we get
\begin{align*}
 K_2 \leq   X^{O(\eta^2)} \sum_{y \leq B} \langle\alpha_{2,y} |\Delta k_{Z_2^2,1} |\alpha_{2,y}\rangle\precprec X^{O(\eta^2)} B (D B^{1/2} + B Z_2).
\end{align*}
Then the total error term is up to a factor of $A^{O(\eta^2)}$
\begin{align*}
     B Z^{1/2} Z_0^\theta ( Z_1^{1/2} + D^{1/2}) (D^{1/2} B^{3/4} + B Z_2^{1/2}).
\end{align*}
Proposition \ref{prop:typeI2tech} now follows with the choice $Z_1= D$ and $Z_2 = 1+ D B^{-1/2}$, which balances the terms. \qed

\subsection*{Acknowledgements}
The authors are grateful to James Maynard for helpful discussions and providing the espresso machine. We are also grateful to Alex Pascadi and Jared Duker Lichtman for helpful discussions. The project has
received funding from the European Research Council (ERC) under the European Union's Horizon 2020 research and innovation programme (grant agreement No 851318).

\bibliography{SL2bib}
\bibliographystyle{abbrv}

\end{document}